\documentclass[onefignum,onetabnum]{siamart190516}
\usepackage{amsfonts,amsmath,amssymb,booktabs}
\usepackage[normalem]{ulem}
\usepackage{longtable}
\usepackage{booktabs}
\usepackage{multirow}

\newtheorem{assumption}{Assumption}[section]

\newcommand{\psdle}{\preccurlyeq}
\newcommand{\psdge}{\succcurlyeq}

\newcommand{\bx}{\boldsymbol{x}}

\newcommand{\by}{\boldsymbol{y}}
\newcommand{\bZ}{\boldsymbol{Z}}

\newcommand{\bz}{\boldsymbol{z}}
\newcommand{\bX}{\boldsymbol{X}}
\newcommand{\bT}{\boldsymbol{T}}

\newcommand{\bW}{\boldsymbol{W}}
\newcommand{\bL}{\boldsymbol{L}}
\newcommand{\bO}{\boldsymbol{O}}
\newcommand{\bP}{\boldsymbol{P}}
\newcommand{\bQ}{\boldsymbol{Q}}

\newcommand{\bSigma}{\boldsymbol\Sigma}

\newcommand{\bU}{\boldsymbol{U}}
\newcommand{\bV}{\boldsymbol{V}}

\def\reals{\mathbb{R}}
\def\bx{\boldsymbol{x}}

\def\bS{\boldsymbol{S}}
\def\b0{\mathbf{0}}
\def\bP{\boldsymbol{P}}
\def\bQ{\boldsymbol{Q}}
\def\bSigma{\boldsymbol\Sigma}

\def\bU{\boldsymbol{U}}

\def\bC{\boldsymbol{C}}
\def\bv{\boldsymbol{v}}
\def\bA{\boldsymbol{A}}
\def\bY{\boldsymbol{Y}}
\def\bB{\boldsymbol{B}}

\def\ba{\boldsymbol{a}}

\def\bI{\mathbf{I}}

\def\tr{\mathrm{tr}}

\newcommand{\rank}{\operatorname{rank}}
\newcommand{\Proj}{\ensuremath{{\Pi}}}

\def\Sp{\mathrm{Sp}}

\usepackage{algorithmic}
\usepackage{algorithm}

\def\Pr{\mathop{\rm Pr\,\!}\nolimits}

\def\Cov{\mathop{\rm Cov}\nolimits}

\def\vec{\mathop{\rm vec}\nolimits}

\def\tr{\mathop{\rm tr}\nolimits}

\newcommand{\trace}{\operatorname{tr}}

\newcommand{\bM}{\boldsymbol{M}}

\newcommand{\bepsilon}{\boldsymbol{\epsilon}}

\newcommand{\bTheta}{\boldsymbol{\Theta}}
\newcommand{\bLambda}{\boldsymbol{\Lambda}}

\newcommand{\bzero}{\mathbf{0}}

  \definecolor{mycolor1}{rgb}{0.1, 0.5, 0.7}
 
\usepackage{lipsum}
\usepackage{amsfonts}
\usepackage{graphicx}
\usepackage{epstopdf}
\usepackage{algorithmic}
\ifpdf
  \DeclareGraphicsExtensions{.eps,.pdf,.png,.jpg}
\else
  \DeclareGraphicsExtensions{.eps}
\fi

\newsiamremark{remark}{Remark}
\newsiamremark{hypothesis}{Hypothesis}
\crefname{hypothesis}{Hypothesis}{Hypotheses}
\newsiamthm{claim}{Claim}

\headers{Orthogonal trace-sum maximization}{J. Won, T. Zhang, and H. Zhou}

\title{Orthogonal Trace-Sum Maximization: Tightness of the Semidefinite Relaxation and Guarantee of Locally Optimal Solutions\thanks{Submitted to the editors DATE.
The authors are listed in alphabetical order.
}}

\author{Joong-Ho Won\thanks{Department of Statistics, Seoul National University
  (\email{wonj@stats.snu.ac.kr}, 
  ).}
\and Teng Zhang\thanks{Department of Mathematics, University of Central Florida
  (\email{teng.zhang@ucf.edu}).}
\and Hua Zhou\footnotemark[3]\thanks{Departments of Biostatistics and Computational Medicine, University of California, Los Angeles
  (\email{huazhou@ucla.edu}).}}

\usepackage{amsopn}

\ifpdf
\hypersetup{
  pdftitle={Orthogonal Trace-Sum Maximization},
  pdfauthor={J. Won, T. Zhang, and H. Zhou}
}
\fi

\begin{document}

\maketitle

\begin{abstract}
This paper studies an optimization problem on the sum of traces of matrix quadratic forms in $m$  {semi-}orthogonal matrices, which can be considered as a generalization of the synchronization of rotations. While the problem is nonconvex, the paper shows that its semidefinite programming relaxation solves the original nonconvex problems exactly  {with high probability}, under an additive noise model with small noise in the order of $O(m^{1/4})$. In addition, it shows that 
the solution of a nonconvex algorithm considered in Won, Zhou, and Lange [\emph{SIAM J. Matrix Anal. Appl., 2 (2021),  pp. 859–882}] is also its global solution 
with high probability under similar conditions. These results can be considered as a generalization of existing results on phase synchronization.
\end{abstract}



\section{Introduction}
This paper considers the orthogonal trace-sum maximization (OTSM) problem \cite{won2018orthogonal} of estimating $m$ orthogonal matrices $\bO_1, \cdots, \bO_m$ with $\bO_i\in\reals^{d_i\times r}$ from the optimization problem:
\begin{equation}\tag{OTSM}\label{eqn:tracemax}
	\text{maximize} ~~ \sum_{1\leq i, j\leq m} \tr(\bO_i^T \bS_{ij} \bO_j) ~~
	\text{subject~to} ~~ \bO_i \in \mathcal{O}_{d_i,r},~ i=1,\dotsc,m
	,
\end{equation}
where $\bS_{ij}=\bS_{ji}^T \in \mathbb{R}^{d_i\times d_j}$ for $i,j = 1,\dotsc,m$, and $r\le\min_{i=1,\dotsc,m} d_i$, 
and
$\mathcal{O}_{d,r} = \{O \in \mathbb{R}^{d\times r}: \bO^T\bO=\bI_r\}$ is the Stiefel manifold of semi-orthogonal matrices; $\bI_r$ denotes the identity matrix of order $r$. 
%
%

The OTSM problem
has applications in generalized canonical correlation analysis (CCA) \cite{hanafi2006analysis} and Procrustes analysis \cite{gower1975generalized,TenBerge1977orthogonal}. Furthermore, if $d_1=\cdots=d_m=r$, then \eqref{eqn:tracemax} reduces to the problem of synchronization of rotations  \cite{Bandeira20164}, which has wide applications in multi-reference alignment \cite{Bandeira2017},  {cryogenic electron microscopy (cryo-EM)} \cite{doi:10.1137/090767777,ZHANG2017159}, 2D/3D point set registration~\cite{7430328,Cucuringu2012, doi:10.1137/130935458}, and multiview structure from motion \cite{6374980,7785130,7789623}.
\subsection{Related works}
While the OTSM problem is proposed recently in \cite{won2018orthogonal}, it is closely related to many well-studied problems. In particular, its special cases have been studied in the name of angular synchronization, which can be considered as a special case of \eqref{eqn:tracemax} in the complex-valued setting, $\mathbb{Z}_2$ synchronization, and synchronization of rotations, and the OTSM problem itself can also be considered as a special case of the group synchronization problem. 

\textbf{Angular synchronization} The complex-valued OTSM problem with $d_1=\cdots=d_m=1$ is equivalent to a problem called angular synchronization or phase synchronization, which estimates angles $\theta_1, \cdots, \theta_m \in[0,2\pi)$ from the observation of relative offsets $(\theta_i-\theta_j)$ mod $2\pi$. The problem has applications in 
 {cryo-EM}
\cite{Singer2009AngularSB}, comparative biology~\cite{Gao2019}, and many others. To address this problem, Singer~\cite{Singer2009AngularSB} formulates the problem as a nonconvex optimization problem \begin{equation}\label{eq:angular}
\max_{\bx\in\mathbb{C}^m}\bx^*\bC\bx,\,\,\text{s.t. $|x_1|=\cdots=|x_m|=1$},
\end{equation}
where $x_k=e^{i\theta_k}$ for all $1\leq k\leq m$. In fact, \eqref{eq:angular} can be considered as the special case of \eqref{eqn:tracemax} when $d_1=\cdots=d_m=r=2$.

The angular synchronization problem \eqref{eq:angular} has been studied extensively. For example, Singer \cite{Singer2009AngularSB} proposes two methods, by eigenvectors and semidefinite programming respectively. The performance of the method is analyzed using random matrix theory and information
theory. In \cite{Bandeira2017}, Bandeira et al. assume the model $\bC=\bz\bz^*+\sigma\bW$, where $\bz\in\mathbb{C}^m$ satisfies $|z_1|=\cdots=|z_m|=1$ and $\bW\in\mathbb{C}^{m\times m}$ is a Hermitian Gaussian Wigner matrix, and show that if $\sigma\leq \frac{1}{18}m^{\frac{1}{4}}$, then the solution of semidefinite programming approach is also the solution to \eqref{eq:angular}  {with high probability}. 
 {Using a more involved argument and a modified power method, Zhong and Boumal \cite{ZhongBoumal18PhaseSynchNearOptimalBounds} improve the bound in \cite{Bandeira2017} to $\sigma=O(\sqrt{\frac{m}{\log m}})$}.

There is another line of works that solve \eqref{eq:angular} using power methods. In particular, Boumal \cite{Boumal16NonconvexPhaseSynch} investigates a modified power method and shows that the method converges to the solution of \eqref{eq:angular} when $\sigma=O(m^{\frac{1}{6}})$, Liu et al.  \cite{LiuYueSo27PowerPhaseSynch} investigate another generalized power method and prove the convergence for $\sigma=O(m^{\frac{1}{4}})$, and Zhong and Boumal \cite{ZhongBoumal18PhaseSynchNearOptimalBounds} improve the rate to $\sigma=O(\sqrt{\frac{m}{\log m}})$.

There are some other interesting works for the angular synchronization problem that are not based on semidefinite programming or power method. \cite{doi:10.1137/18M1217644} assumes that the pairwise differences are only observed over a graph, studies the landscape of a proposed objective function, and shows that the global minimizer is unique when the associated graph is incomplete and follows the Erd\"{o}s-R\'{e}nyi model. \cite{doi:10.1002/cpa.21750} proposes an approximate message passing (AMP) algorithm, and analyzes its behavior by identifying phases where the problem is easy, computationally hard, and statistically impossible.

\textbf{$\mathbb{Z}_2$ synchronization}  The real-valued OTSM problem with $d_1=\cdots=d_m=1$ is called $\mathbb{Z}_2$-synchronization problem~\cite{10.1093/comnet/cnu050} for   $\mathbb{Z}_2=\{1,-1\}$. For this problem, \cite{Fei2019AchievingTB} shows that the solution of the semidefinite programming method matches the minimax lower bound on the optimal Bayes error rate for the original problem \eqref{eq:angular}.

\textbf{Synchronization of rotations} The OTSM problem with $d_1=\cdots=d_m=r>2$ is called ``synchronization of rotations'' in some literature. This special case has wide applications in  graph realization and point cloud registration, multiview Structure from Motion (SfM)  \cite{6374980,7785130,7789623},  common lines in cryo-EM \cite{doi:10.1137/090767777},  orthogonal least squares~\cite{ZHANG2017159}, and 2D/3D point set registration~\cite{7430328}. \cite{10.1093/imaiai/iat006} studies the problem from the perspective of manifold optimization and derive the Cram\'{e}r-Rao bounds, which is the lower bounds of the variance of any unbiased estimator. 
\cite{THUNBERG2017243} proposes a distributed algorithm with theoretical guarantees on convergence. 
 \cite{10.1093/imaiai/iat005} discusses a method to make the estimator in \eqref{eqn:tracemax}  more robust to outlying observations. Another robust algorithm based on maximum likelihood estimator is proposed in \cite{6760038}. As for the theoretical properties, \cite{Bandeira20164} analyzes a semidefinite program approach that solves the problem approximately, and studies its approximation ratio. \cite{liu2020unified} investigates a generalized power method for this problem. A recent manuscript \cite{ling2021improved} follows the line of \cite{Bandeira2017,Boumal16NonconvexPhaseSynch,LiuYueSo27PowerPhaseSynch,ZhongBoumal18PhaseSynchNearOptimalBounds} and proves that the original problem and the relaxed problem   have the same solution when $\sigma\leq O(\frac{\sqrt{m}}{d+\sqrt{d}\log m})$. 

\textbf{Group synchronization} The OTSM problem can also be considered as a special case of the group synchronization problem, which recovers a vector of
 elements in a group, given noisy pairwise measurements of the relative elements $g_ug_v^{-1}$. The OTSM problem is the special case when the group is $\mathcal{O}_{d,r}$, the set of orthogonal matrices. \cite{abbe2017group} studies the properties of weak recovery when the elements are from a generic compact group and the underlying graph of pairwise observations is the $d$-dimensional grid.   \cite{doi:10.1002/cpa.21750} proposes an approximate message passing (AMP) algorithm for solving synchronization problems over a class of compact groups. \cite{doi:10.1137/16M1106055} generates the estimation from compact groups to the class of Cartan motion groups, which includes the important special case of rigid motions by applying the compactification process. 
 \cite{6875186}  assumes that measurement graph is sparse and there are corrupted observations, and show that minimax recovery rate depends almost exclusively on the edge sparsity of the measurement graph irrespective of other graphical metrics.

\subsection{Our contribution}
The main contribution of this work is the study of the OTSM problem under an additive model of Gaussian noise. The main results are two-folds: First, we propose a semidefinite programming approach for solving \eqref{eqn:tracemax} and show that it solves \eqref{eqn:tracemax} exactly when the size of noise is bounded. Second, we show that any local minimizer of \eqref{eqn:tracemax} is the global minimizer, when the noise size is bounded and a technical assumption is satisfied. 
 {These} results can be considered as a generalization of \cite{Bandeira2017} from angular synchronization to the OTSM problem. 

\section{The OTSM problem}
\subsection{Model assumption}
In this work, we assume the MAXBET model of generating $\bS_{ij}$, which 
 {postulates the existence of $\{\bTheta_i\}_{1\leq i\leq m}$ and $\{\bW_{ij}\}_{1\leq i\neq j\leq m}$} 
such that 
$\bTheta_i\in\mathcal{O}_{d_i,r}$ for all $1\leq i\leq m$, and
\begin{equation}
	\text{$\bS_{ij}=\bTheta_i\bTheta_j^T+\bW_{ij}$, where $\bW_{ij}=\bW_{ji}^T$ for all $1\leq i,j\leq m$.}\label{eq:model}\tag{MAXBET}
\end{equation} 
In this model, $\bTheta_i\bTheta_j^T$ is considered as the ``clean measurement of relative elements'' and $\bW_{ij}$ is considered as an additive noise. 
%
{%
This is a natural model for the generalized CCA in \cite{won2018orthogonal}.
Consider a latent variable model in which a latent variable $\bz \in \mathbb{R}^r$ has zero mean and covariance matrix $\bI_r$, and an observation in the $i$th group is given by $\ba_i = \bTheta_i\bz + \bepsilon_i \in \mathbb{R}^{d_i}$, $i=1, \dots, m$, with the noise $\bepsilon_i$ is uncorrelated with $\bz$ and $\bepsilon_j$, $j\neq i$.
If the noise covariance is $\tau\bI_{d_i}$, then the auto-covariance of group $i$ is $\bSigma_{ii} + \tau\bI_{d_i}$.
The (population) cross-covariance matrix between groups $i$ and $j$ is $\bSigma_{ij} = \bTheta_i\bTheta_j^T$. 
The generalized CCA \cite{TenBerge1977orthogonal,won2018orthogonal} seeks (semi-) orthogonal matrices $\{\bO_i \in \mathcal{O}_{d_i, r}\}$ such that the expected inner product between matrices $\bO_i^T\ba_i$ and $\bO_j^T\ba_j$ is summed and maximized for each pair $(i, j)$, which is $\sum_{i,j}\tr(\bO_i^T\bSigma_{ij}\bO_j)$.
Also note that $\mathbb{E}[\langle \bO_i^T\ba_i, \bO_i^T\ba_i \rangle] = \tr(\bO_i^T\bSigma_{ii} \bO_i) + \text{const}$. 
If we assume that $\{\bTheta_i\}$ are (semi-) orthogonal, then this problem is precisely \eqref{eqn:tracemax}, and the forthcoming \Cref{prop:noiseless} shows that the population version of this generalized CCA  recovers precisely the transformations $\{\bTheta_i\}$ of the latent variable $\bz$.
Now let us turn to the practical setting.
Appealing to the law of large numbers, the sample estimate of $\bSigma_{ij}$ can then be written as $\bS_{ij} = \bSigma_{ij} + \bW_{ij} = \bTheta_i\bTheta_j^T + \bW_{ij}$. 
A statistical interest is whether $\{\bTheta_i\}$ can be precisely estimated by solving the sample version of \eqref{eqn:tracemax}.
Model \eqref{eq:model} is also standard for synchronization problems, such as synchronization of rotations \cite{10.1093/imaiai/iat005,10.1093/imaiai/iat006} and group synchronization \cite{abbe2017group,doi:10.1002/cpa.21750}.
}

In some applications \cite{TenBerge1977orthogonal,hanafi2006analysis}, it is also natural to assume the MAXDIFF model that 
ignores the auto-covariance terms:
\begin{equation}\label{eq:model2}\tag{MAXDIFF}
\bS_{ii}=\bzero
\quad\text{and}\quad
\bS_{ij}=\bTheta_i\bTheta_j^T+\bW_{ij},
~i \neq j
.
\end{equation}
In this work, we will present our main results based on the MAXBET model, and discuss the MAXDIFF model in the remarks.

When there is no noise in either MAXBET or MAXDIFF model, setting $\bO_i=\bTheta_i$, $i=1, \dotsc, m$, solves problem \eqref{eqn:tracemax}. 
The proof is deferred to Section~\ref{sec:proof_prop:noiseless}.
\begin{proposition}\label{prop:noiseless} In the noiseless case ($\bW_{ij}=\bzero$ for all $i, j$), 
$(\bO_1,\allowbreak \dotsc, \allowbreak \bO_m) = (\bTheta_1, \dotsc, \bTheta_m)$ globally solves \eqref{eqn:tracemax}
under the model \eqref{eq:model} or \eqref{eq:model2}.
\end{proposition}

However, in the presence of noise, \Cref{prop:noiseless} does not hold and problem \eqref{eqn:tracemax} is difficult to solve. To establish theoretical guarantees for the noisy setting, we investigate two approaches, one is based on semidefinite programming and the other one is based on finding local optimal solutions of \eqref{eqn:tracemax}.

\subsection{Approach 1: Semidefinite programming relaxation}

While the problem \eqref{eqn:tracemax} is nonconvex and difficult to solve,
we can relax it to a convex optimization problem based on semidefinite programming that can be solved efficiently. In fact, semidefinite programming-based approaches have been proposed and analyzed for the problem of angular synchronization \cite{Singer2009AngularSB,Bandeira2017,ZhongBoumal18PhaseSynchNearOptimalBounds} and synchronization of rotations \cite{Bandeira20164}, and our proposed method can be considered as a generalization of these existing methods.

The argument of the relaxation is as follows. Let $D=\sum_{i=1}^md_i$, 
\begin{equation}\label{eqn:blockS}
\bS=\begin{bmatrix}
    \bS_{11} & \bS_{12} & \dots  & \bS_{1m} \\
    \bS_{21} & \bS_{22} &  & \bS_{2m} \\
    \vdots &  & \ddots& \vdots \\
   \bS_{m1} & \bS_{m2} &\cdots  &\bS_{mm}
\end{bmatrix}\in\reals^{D\times D},  \,\,\,\text{and} \,\,\,{\bO}=\left[ \begin{array}{c}
{\bO}_1\\
\vdots\\
{\bO}_m\\
 \end{array} \right]\in\reals^{D\times r},
\end{equation}
then by setting $\bU=\bO\bO^T$, 
the problem
\eqref{eqn:tracemax} 
is equivalent to finding
\begin{equation}\label{eq:tracesum_problem1}
\tilde{\bU} \!=\! \arg\max\{\tr(\bS\bU)\!: \bU\psdge \mathbf{0},\,\rank(\bU)\!=\!r,\,\bU_{ii}\psdle\bI,\,\tr(\bU_{ii})=r,\,i\!=\!1,\dotsc,m\}
\end{equation}
for $\bU\in\reals^{D\times D}$ such that $\bU=\bU^T$,
which
can be relaxed to solving
\begin{equation}\label{eq:tracesum_relaxed}\tag{SDP}
\max_{\bU\in\reals^{D\times D}, \bU=\bU^T}\langle\bS, \bU\rangle ~~ \text{subject to $\bU\psdge \mathbf{0}, \bU_{ii}\psdle\bI, \tr(\bU_{ii})=r$,}
\end{equation}
where $\bM \psdge \bzero$ (resp. $\bM \psdle \bzero$) means that a matrix $\bM$ is positive (resp. negative) semidefinite.
 {%
If a solution $\hat{\bU}$ to problem \eqref{eq:tracesum_relaxed} has rank $r$, then we can set $\tilde{\bU}=\hat{\bU}$, which can be decomposed to $\hat{\bU}=\tilde{\bV}\tilde{\bV}^T$ where $\tilde{\bV}\in\mathbb{R}^{D\times r}$. Write $\tilde{\bV}=[\tilde{\bV}_1^T, \dotsc, \tilde{\bV}_m^T]^T$, then $\tilde{\bV}_i\in\mathcal{O}_{d_i,r}$ for all $1\leq i\leq m$ and $(\tilde{\bV}_1, \dotsc, \tilde{\bV}_m)$ \emph{globally} solves problem \eqref{eqn:tracemax}.%
}

This work shows that if the noises $\bW_{ij}$ are ``small'', then the solutions of problems \eqref{eqn:tracemax} and \eqref{eq:tracesum_relaxed} are equivalent in the sense that
 {$\hat{\bU}=\tilde{\bV}\tilde{\bV}^T$ with $\tilde{\bV}$ rank-$r$,} 
hence the convex relaxation is tight.
{Furthermore, each $\tilde{\bV}_i$ converges to $\bTheta_i$ as $m\to\infty$, as desired for CCA applications.}

\subsection{Approach 2: finding local optimal solutions of \eqref{eqn:tracemax}}
While the semidefinite programming approach is convex and can be solved efficiently, it has a large computational cost when $D$ is large, and solving the original nonconvex problem \eqref{eqn:tracemax} 
 {without lifting the variable (from $\bO$ to $\bU$)}
is more efficient. A natural question is, is there any guarantee on the goodness of local minimizers of \eqref{eqn:tracemax} and when is a critical point of problem \eqref{eqn:tracemax} globally optimal? 

Using the optimality conditions for the convex relaxation \eqref{eq:tracesum_relaxed}, Won, Zhou, and Lange \cite{won2018orthogonal} study  conditions for a first-order critical point of problem \eqref{eqn:tracemax} to be globally optimal. 
Specifically, 
the first-order necessary condition for local optimality of \eqref{eqn:tracemax} is
\begin{equation}\label{eqn:firstorder}
	\bO_i\bLambda_i = \sum_{j=1}^m \bS_{ij}\bO_j, \quad i=1,\dotsc,m,
\end{equation}
for some symmetric matrix $\bLambda_i$. The latter matrix is the Lagrange multiplier associated with the constraint $\bO_i \in \mathcal{O}_{d_i,r}$, and has a representation $\bLambda_i = \sum_{j\neq i}\bO_i^T\bS_{ij}\bO_j$.
If $\tau_i$ is the smallest eigenvalue of $\bLambda_i$,
then a first-order critical point $(\bO_1,\dotsc,\bO_m)$ is a global optimum of \eqref{eqn:tracemax} if
\begin{equation}\label{eqn:certificate}
	\bL :=
	\begin{bmatrix} 
	\bO_1\bLambda_1\bO_1^T+\tau_1(\bI_{d_1}-\bO_1\bO_1^T)
		 &     \\
	     &   \ddots & \\
	     &          &
		\bO_m\bLambda_m\bO_m^T+\tau_m(\bI_{d_m}-\bO_m\bO_m^T)
	\end{bmatrix}
	- \bS
	\psdge \bzero
	.
\end{equation}
Condition \eqref{eqn:certificate} requires finding a first-order critical point. A block relaxation-type algorithm that converges to a first-order point is considered in \cite{won2018orthogonal}.
\textit{A priori} conditions on the data matrix $\bS$ under which $\bL \psdge \bzero$ has remained an open question. %

This paper shows that if the noises $\bW_{ij}$ are ``small'', then  all ``not too bad'' 
critical points (for precise definition, see Section \ref{sec:mainresults:local})
of \eqref{eqn:tracemax} satisfy condition \eqref{eqn:certificate}, hence are globally optimal.
{Furthermore, each $\bO_i$ converges to $\bTheta_i$ as $m\to\infty$, up to a common phase shift, as desired for CCA applications.}
The meanings of ``small'' and ``not too bad'' will be made clear in the sequel.


\subsection{Notation}

This work sometimes divides a matrix $\bX$ of size $D\times D$ into $m^2$ submatrices, such that the 
 {$(i,j)$}
block is a $d_i\times d_j$ submatrix. We use $\bX_{ij}$ or $[\bX]_{ij}$ to denote this submatrix. Similarly, sometimes we divide a matrix of $\bY\in\reals^{D\times r}$ or a vector $\by\in\reals^D$ into $m$ submatrices or $m$ vector, where the $i$-th component, denoted by $\bY_i$, $[\bY]_i$ or $\by_i$, $[\by]_i$, is a matrix of size $d_i\times r$ or a vector of length $d_i$.

For any matrix $\bX$, we use $\|\bX\|$ to represent its operator norm and $\|\bX\|_F$ to represent its Frobenius norm. In addition, $\bP_{\bX}$ represents an orthonormal matrix whose column space is the same as $\bX$, $\bP_{\bX^\perp}$ is an orthonormal matrix whose column space is the orthogonal complement of the column space of $\bX$, $\Proj_{\bX}=\bP_{\bX}\bP_{\bX}^T$ is the projector to the column space of $\bX$, and $\Proj_{\bX^\perp}$ is the projection matrix to the orthogonal complement of the column space of $\bX$. If $\bY\in\reals^{n\times n}$ is symmetric, we use $\lambda_1(\bY)\geq \lambda_2(\bY)\geq \cdots\geq \lambda_n(\bY)$ to denote its eigenvalues in descending order.

\section{Main results}\label{sec:mainresults}
In this section, we present our main results. The first main result, Theorem~\ref{thm:main}, shows that if the noises $\bW_{ij}$ are ``small'', then the convex relaxation in \eqref{eq:tracesum_relaxed} solves the original problem \eqref{eqn:tracemax} exactly.  The second main result, Theorem~\ref{thm:main2}, shows that if the noises $\bW_{ij}$ are ``small'', then  
the critical points obtained by the algorithm in \cite{won2018orthogonal}
of \eqref{eqn:tracemax} is globally optimal
if an additional technical assumption is satisfied.

\subsection{Theoretical guarantees on the SDP approach}
This section provides conditions that if the noises $\bW_{ij}$ are ``small'', then the solution of problem \eqref{eq:tracesum_relaxed} has rank $r$ and is equivalent to the solution of the problem \eqref{eqn:tracemax}, in the sense that
 {$\hat{\bU}=\tilde{\bV}\tilde{\bV}^T$ with $\tilde{\bV}$ rank-$r$,}
hence the convex relaxation is tight.
 
 We begin with two deterministic conditions on $\bW$ in \Cref{thm:main} and \Cref{cor:probablistic}, show that the condition holds with high probability under a Gaussian model in \Cref{cor:probablistic_gaussian}, and a statement on the consistency of the solution in \Cref{cor:consistency1}. The statement of the first deterministic theorem is as follows:
\begin{theorem}\label{thm:main}
{If $m\geq \|\bW\|(4\sqrt{r}+1)+1$}, and $\bW\in\reals^{D\times D}$ is small in the sense that
\begin{eqnarray}\label{eq:condition_deterministic}
m &>& 4m\frac{2\left(\max_{1\leq i\leq m}\|[\bW{\bTheta}]_i\|_F+4\|\bW\|^2\sqrt{\frac{r}{m}}\right)}{m-\|\bW\|(4\sqrt{r}+1)-1} \nonumber \\
& & \quad +2\left(\max_{1\leq i\leq m}\|[\bW{\bTheta}]_i\|_F+4\|\bW\|^2\sqrt{\frac{r}{m}}\right)+8\|\bW\|\sqrt{\frac{r}{m}}+2\|\bW\|,\end{eqnarray}
then the solutions of \eqref{eqn:tracemax} and relaxation \eqref{eq:tracesum_relaxed} are equivalent in the sense that 
{a solution $\hat{\bU}$ to  \eqref{eq:tracesum_relaxed} also solves \eqref{eq:tracesum_problem1}.}
\end{theorem}
The proof of \Cref{thm:main} will be presented in Section~\ref{sec:proof_thm:main}. While the condition \eqref{eq:condition_deterministic} is rather complicated,  we expect that it holds for large $m$ when $\|\bW\|$ and $\max_{i=1,\dotsc,m}\|(\bW\bTheta)_i\|_F$  grow slowly as $m$ increases.  To prove this idea rigorously, we introduce the notion of $\bTheta$-discordant noise, which is inspired by the notion of ``$z$-discordant matrix'' in \cite[Definition 3.1]{Bandeira2017}. 
\begin{definition}[$\bTheta$-discordance]\label{def:discordance}
	Let $\bTheta=(\bTheta_1,\dotsc,\bTheta_m)\in\times_{i=1}^m\mathcal{O}_{d_i,r}$. Recall $D=\sum_{i=1}d_i$. A matrix $\bW$ is said to be $\bTheta$-discordant if it is symmetric and satisfies 
	$
		\|\bW\| \le 3\sqrt{D}
	$
	and
	$
		\max_{i=1,\dotsc,m}\|[\bW\bTheta]_i\|_F 
		\le 3\sqrt{Dr\log m}
		.
	$
\end{definition}
Based on the definition of $\bTheta$-discordant noise, The next corollary is a deterministic, non-asymptotic statement that simplifies the condition \eqref{eq:condition_deterministic} in \Cref{thm:main}. Its proof is deferred to Section~\ref{sec:proof_cor:probablistic}.
\begin{corollary}\label{cor:probablistic}
Let $d=D/m$. 
If $m\geq 8$ and $\sigma^{-1}\bW$ is $\bTheta$-discordant for 
\begin{equation}\label{eq:cor:discordant}
	\sigma \le \frac{m^{1/4}}{60\sqrt{dr}},
\end{equation}
then condition \eqref{eq:condition_deterministic} holds, and the solutions of \eqref{eqn:tracemax} and \eqref{eq:tracesum_relaxed} are equivalent.
\end{corollary}
Next, we apply a natural probabilistic model and investigate the $\bTheta$-discordant property. In particular, we follow \cite{Boumal16NonconvexPhaseSynch,Bandeira2017,ZhongBoumal18PhaseSynchNearOptimalBounds} and use an additive Gaussian noise model {to generate the symmetric noise matrix $\bW$}:
\begin{equation}\label{eq:noise_model}
{\text{Upper triangular part of $W \in \mathbb{R}^{D \times D}$ is elementwisely i.i.d. sampled from $N(0,\sigma^2)$.}}
\end{equation}
For this model, we have the following corollary that shows if 
 {$\sigma \le O(\frac{m^{1/4}}{\sqrt{dr}})$,}
then \eqref{eq:condition_deterministic} holds with high probability. Its proof is deferred to Section~\ref{sec:proof_cor:probablistic_gaussian}.
\begin{corollary}\label{cor:probablistic_gaussian}
	Assume the additive Gaussian noise model in \eqref{eq:noise_model}, $m\ge 3$ or $m \ge 2$ and $\min_{i=1}^md_i \ge 6$, then with probability at least $1-1/m-2\exp\left(-\frac{(3-2\sqrt{2})^2}{4}D\right)$, $\bW$ satisfies the $\bTheta$-discordant property. 
	
	As a result, if $\sigma \le \frac{m^{1/4}}{60\sqrt{dr}}$ and $m\ge 8$, then with the same probability, the condition \eqref{eq:condition_deterministic} holds, and the solutions of \eqref{eqn:tracemax} and \eqref{eq:tracesum_relaxed} are equivalent.
\end{corollary}

\begin{remark}
The assumption $m\geq 8$ in \Cref{cor:probablistic} can be relaxed, but with a different constant factor in the upper bound of $\sigma$ in \eqref{eq:cor:discordant}. For example, if $m\geq 3$ is assumed, then we need 	$\sigma \le \frac{m^{1/4}}{124\sqrt{dr}}$. 
\end{remark}
\begin{remark}
The result in this section can be naturally adapted to the 
MAXDIFF
model. The main intermediate results for the proof of \Cref{thm:main} 
 {given in Section \ref{sec:proof_thm:main}},
including \Cref{lemma:equavalency} and \Cref{lemma:noisy2}, still hold with $\bW_{ii}=0$. While the estimations in \Cref{lemma:pertubation} do not hold, following the steps 
 {given at the end of Section \ref{sec:MAXBET}},
we are still able to obtain similar bounds on the difference between $\tilde{\bV}$ and 
 {$\bTheta$.}
In summary, we are able to obtain parallel results to \Cref{thm:main} and \Cref{cor:probablistic} for the 
MAXDIFF
setting. In particular, if $\bW$ is generated using the model in \Cref{cor:probablistic}, then the solutions of \eqref{eqn:tracemax} and the 
MAXDIFF
problem are equivalent with probability at least  $1-1/m-2\exp\left(-\frac{(3-2\sqrt{2})^2}{4}D\right)$, if $\sigma\leq \frac{m^{1/4}}{120\sqrt{dr}}$ and $m\geq 10$. 
    This, more restrictive, bound under the MAXDIFF model is expected since MAXDIFF utilizes \emph{less} information on the clean signal $\bTheta$ for the same number of measurements.
\end{remark}

Following the proof of \Cref{thm:main}, we have 
 {a consistency result, i.e.,}
that the solution of \eqref{eq:tracesum_relaxed} recovers the true signal $\bTheta$ if $m$ is sufficiently large:
\begin{corollary}\label{cor:consistency1}
Assuming the conditions in \Cref{cor:probablistic}, then the solution of \eqref{eq:tracesum_relaxed}, $\tilde{\bU}$, admits a decomposition $\tilde{\bU}=\tilde{\bV}\tilde{\bV}^T$ with $\tilde{\bV}\in\reals^{D\times r}$, such that
\begin{align}\label{eq:consistency1}
&\max_{i=1,\cdots m}\|\tilde{\bV}_i-\bTheta_i\|_F\leq \frac{2\left(3\sigma\sqrt{dmr\log m}+36\sigma^2d\sqrt{rm}\right)}
{m-3\sigma\sqrt{dm}(4\sqrt{r}+1)-1}.
\end{align}
Thus, if $\sigma = o\left(\frac{m^{1/4}}{\sqrt{dr}}\right)$, then  $\max_{i=1,\cdots m}\|\tilde{\bV}_i-\bTheta_i\|_F \to 0$ as $m\to\infty$.
\end{corollary}
\begin{remark}
For the MAXDIFF model, \eqref{eq:consistency1} is replaced with
\[
\max_{i=1,\cdots m}\|\tilde{\bV}_i-\bTheta_i\|_F\leq\frac{6\sigma\sqrt{dmr\log m}+\frac{72\sigma^2dm\sqrt{rm}}{m-2}}{m-\frac{12\sigma\sqrt{dm^3r}}{m-2}}.
\]
The bound follows from the discussion of Lemma~\ref{lemma:pertubation} in the MAXDIFF setting. If $\sigma = o\left(\frac{m^{1/4}}{\sqrt{dr}}\right)$, then  $\max_{i=1,\cdots m}\|\tilde{\bV}_i-\bTheta_i\|_F \to 0$ as $m\to\infty$. 
\end{remark}

\subsection{Theoretical guarantees on local optimal solutions}\label{sec:mainresults:local}
This section pre-sents the condition on the size of the noise $\bW$ that ensures locally optimal points of problem  \eqref{eqn:tracemax} are in fact globally optimal.  We begin with two deterministic conditions on $\bW$ in \Cref{thm:main2} and \Cref{cor:discordant},  show that the condition in \Cref{cor:discordant} holds with high probability under the additive Gaussian model~\eqref{eq:noise_model} in \Cref{cor:gaussian}, and establish the consistency in \Cref{prop:consistencylocal}. 

Recall that the first-order necessary condition for local optimality of \eqref{eqn:tracemax} is given in equation \eqref{eqn:firstorder}. 
The associated Lagrange multiplier is symmetric:
\begin{equation}\label{eqn:lagrange0}
	\bLambda_i = \bO_i^T\big(\sum_{j=1}^m\bS_{ij}\bO_j\big) 
	= \big(\sum_{j=1}^m\bS_{ij}\bO_j\big)^T\bO_i
	.
\end{equation}
%
%
We call a point $\bO=(\bO_1,\dotsc,\bO_m)$ that satisfies the first-order local optimality condition a \emph{critical point}.
It has been shown that a necessary condition for global optimality of a critical point is that the $\bLambda_i$ in equation \eqref{eqn:lagrange0} is symmetric and positive semidefinite for all $i$ \cite[Proposition 3.1]{won2018orthogonal}.

%
{%
Thus we introduce the following qualification of a critical point.
\begin{definition}[Candidate critical point]\label{def:qual}
A critical point $\bO=(\bO_1,\dotsc,\bO_m)$ is called a \emph{candidate critical point} if the associate Lagrange multipliers $\bLambda_i$ as defined in equation \eqref{eqn:lagrange0} are symmetric and positive semidefinite for all $i=1, \dotsc, m$.
\end{definition}
An algorithm for finding a candidate critical point is also provided in \cite{won2018orthogonal}; see Algorithm 4.1 and the proof of Proposition 3.1 there.
In the sequel, we assume:
}%
\begin{assumption}\label{assumption}
Point $\bO=(\bO_1, \dotsc, \bO_m)\in\times_{i=1}^m\mathcal{O}_{d_i,r}$ is a 
{candidate}
critical point such that
	$
		\tr(\bTheta^T\bS\bTheta) \le \tr(\bO^T\bS\bO)
		.
	$
\end{assumption}
{%
A similar qualification 
is used in \cite[Proposition 4.5]{Bandeira2017} to prove global optimality of qualified critical points in the angular synchronization problem.
Note that a global optimizer of \eqref{eqn:tracemax} satisfies \Cref{assumption}. 
Therefore there is at least one candidate critical point satisfying this assumption. 
Verification of this assumption is discussed shortly after we state the main theorem of this section:
}
\begin{theorem}\label{thm:main2}
If noise $\bW$ is small in the sense that
	\begin{align}\label{eqn:noiselevel}
	m &\ge
    \|\bW\|(4\sqrt{r}+1)
    + \max_{1\le i\le m}\|[\bW\bTheta]_i\|_F + 4\|\bW\|^2\sqrt{\frac{r}{m}}
    \nonumber\\
    & \quad
    + \frac{2{m}(\max_{1\le i\le m}\|[\bW\bTheta]_i\|_F + 4\|\bW\|^2\sqrt{\frac{r}{m}})}{m - 4\|\bW\|\sqrt{r}}
    + 16\|\bW\|^2\frac{r}{m}
	,
	\end{align}
	then all the points $\bO=(\bO_1,\dotsc,\bO_m)$ satisfying Assumption \ref{assumption} are \emph{global optima} of \eqref{eqn:tracemax}.
\end{theorem}
{
The proof of this theorem, deferred to Section \ref{sec:proofmain2:thm}, essentially shows that under the condition \eqref{eqn:noiselevel} for the noise level,
\Cref{assumption} implies that  the  certificate \eqref{eqn:certificate} of global optimality for $\bO$ holds. 
Thus if \eqref{eqn:certificate}  does  not  hold,  then  the  corresponding candidate  critical  point  cannot  satisfy  \Cref{assumption},  provided  that \eqref{eqn:noiselevel} is  true.
The forthcoming \Cref{cor:discordant,cor:gaussian} assert that condition \eqref{eqn:noiselevel} holds with high probability if noise variance is small.
Therefore certificate \eqref{eqn:certificate} is almost \emph{necessary and sufficient} for global optimality of a candidate critical point in this regime, and can be used to verify \Cref{assumption}. 
In the simulation study presented in Appendix, it is numerically demonstrated that \Cref{assumption} is satisfied by candidate points generated by the proximal block ascent algorithm in \cite{won2018orthogonal} for a wide range of noise variances, even if condition \eqref{eqn:noiselevel} is not satisfied.
}

{%
The following corollary is a deterministic, non-asymptotic statement that simplifies  condition \eqref{eqn:noiselevel} using the notion of $\bTheta$-discordance (\Cref{def:discordance}).
The idea is similar to \eqref{eq:condition_deterministic}. 
The left-hand side of condition \eqref{eqn:noiselevel} dominates the right-hand side as $m\to\infty$, if $\|\bW\|$ and $\max_{i=1,\dotsc,m}\|(\bW\bTheta)_i\|_F$ are bounded or increase slowly as $m$ increases.
Thus, we can expect that a candidate critical point is globally optimal if noise variance $\sigma$ is small and the number of observations $m$ is large: %
}


\begin{corollary}\label{cor:discordant}
Let $d=D/m$. Suppose that $m\geq 2$, 	\begin{equation}\label{eq:cor:discordant2}
	\sigma \le \frac{m^{1/4}}{31\sqrt{dr}},
	\end{equation}
	and $\sigma^{-1}\bW$ is $\bTheta$-discordant, then \eqref{eqn:noiselevel} holds, and all the points $\bO=(\bO_1,\dotsc,\bO_m)$ satisfying Assumption \ref{assumption} are global optima of \eqref{eqn:tracemax}.
\end{corollary}
The proof is deferred to Section \ref{sec:proof_discordant}.

Finally, since \Cref{cor:probablistic_gaussian} shows that $\bW$ in the additive Gaussian noise model \eqref{eq:noise_model} is $\bTheta$-discordant after scaling by $\sigma$, \Cref{cor:discordant} implies the following result on the probabilistic model:
\begin{corollary}\label{cor:gaussian}
	Suppose the additive Gaussian noise model in \eqref{eq:noise_model}.
	If $\sigma \le \frac{m^{1/4}}{31\sqrt{rd}}$ and $m\ge 3$ or $m \ge 2$ and $\min_{i=1,\dotsc, m}d_i \ge 6$, then with probability at least $1-1/m-2\exp\left(-\frac{(3-2\sqrt{2})^2}{4}D\right)$, all the points $\bO=(\bO_1,\dotsc,\bO_m)$ satisfying Assumption \ref{assumption} are global optima of \eqref{eqn:tracemax}.
\end{corollary}
\begin{remark}\label{rem:maxdifflocal0}
    The upper bound of $\sigma$ in the RHS of \eqref{eq:cor:discordant2} can be made smaller if $m$ increases. For example, if we have $m \ge 4$, then \eqref{eq:cor:discordant2} can be relaxed to $\sigma \le \frac{m^{1/4}}{29\sqrt{dr}}$; if $m \ge 9$, $\sigma \le \frac{m^{1/4}}{26\sqrt{dr}}$ suffices.
\end{remark}

\begin{remark}\label{rem:maxdifflocal}
    If instead the MAXDIFF model is assumed, the present analysis holds for $m \ge 4$ and \eqref{eq:cor:discordant2} replaced with $\sigma \le \frac{m^{1/4}}{64\sqrt{dr}}$.
    This is a worse bound as opposed to $\frac{m^{1/4}}{29\sqrt{dr}}$ for MAXBET (See Remark~\ref{rem:maxdifflocal0}).
    To obtain the same bound as \eqref{eq:cor:discordant2}, we need $m \ge 9$; see Section \ref{sec:proof:maxdiff}.
    Similar to the SDP relaxation, the more restrictive bound in the MAXDIFF model is expected since MAXDIFF utilizes less information on the clean signal $\bTheta$ for the same number of measurements.
\end{remark}

{ %
    The following consistency result is a by-product of the proof of \Cref{thm:main2}. 
    Recall that problem \eqref{eqn:tracemax} is invariant to ``simultaneous rotation,'' i.e., post-multiply-ing a fully orthogonal matrix $\bQ\in\mathcal{O}_{r,r}$ to $\bO_i$'s (see, e.g., \cite[Eq. (8.2)]{wang2013orientation}):
    \begin{corollary}\label{prop:consistencylocal}
            Let $(\bO_1, \dotsc, \bO_m)\in\times_{i=1}^m\mathcal{O}_{d,r}$ be a 
            point
            satisfying Assumption \ref{assumption}. If the noise $\sigma^{-1}\bW$ is $\bTheta$-discordant and  
            $m > 144\sigma^2dr$, we have an estimation error
\[      	   \min_{\bQ\in\mathcal{O}_{r,r}}  \max_{1\le i\le m}\|\bO_i\bQ - \bTheta_i\|_F \le
	        \frac{2\left(3\sigma\sqrt{\frac{dr\log m}{m}} + 36\sigma^2d\sqrt{\frac{r}{m}}\right)}{1 - 12\sigma\sqrt{\frac{dr}{m}}}
	        .
\]	   
          
    \end{corollary}
Thus if $\sigma = o\left(\frac{m^{1/4}}{\sqrt{dr}}\right)$ then we have $
\min_{\bQ\in\mathcal{O}_{r,r}} \max_{1\le i\le m}\|\bO_i\bQ - \bTheta_i\|_F \to 0$
as $m\to\infty$, as desired.
\begin{remark}
            If the MAXDIFF model is assumed,  $m > 2$, and $m^{3/2} - 2m^{1/2} - 12\sigma\sqrt{dr}m - 3 > 0$, then 
\[            
	   \min_{\bQ\in\mathcal{O}_{r,r}} \max_{1\le i\le m}\|\bO_i\bQ - \bTheta_i\|_F \le
        \frac{2\left(3\sigma\sqrt{\frac{dr\log m}{m}} + 36\sigma^2\frac{d\sqrt{r}}{\sqrt{m}-2/\sqrt{m}}\right)}{1 - 12\sigma\frac{\sqrt{dr}}{\sqrt{m}-2/\sqrt{m}}-\frac{3}{m}}
        .
\]	     
\end{remark}
}



\subsection{Comparison with existing works}
Our results generalize the work \cite{Bandeira2017} on angular synchronization, which analyzes the setting $d=r=1$ with complex values. In particular, \Cref{thm:main}, \Cref{cor:probablistic}, \Cref{cor:probablistic_gaussian}, and \Cref{cor:gaussian} are generalizations of Lemma 3.2, Theorem 2.1, and Proposition 4.5 in \cite{Bandeira2017} respectively. \Cref{cor:probablistic} is similar to Lemma 3.2 in \cite{Bandeira2017}, in the sense that both results establish deterministic conditions such that the original problem and the relaxed problem have the same solutions, under a ``discordant'' condition. In addition, \Cref{cor:probablistic_gaussian} is a generalization of~\cite[Theorem 2.1]{Bandeira2017}, in the sense that both results establish upper bounds on the size of noise $\sigma$ under an additive Gaussian model. At last, both \Cref{cor:gaussian} and Proposition 4.5 in \cite{Bandeira2017} show that local solutions satisfying an assumption are global optima.

Theorem 2.1 and Proposition 4.5 in \cite{Bandeira2017} require $\sigma\leq \frac{1}{18}m^{\frac{1}{4}}$. In comparison, \Cref{cor:probablistic_gaussian} and \Cref{cor:gaussian} require $\sigma\leq \frac{1}{60}m^{\frac{1}{4}}$ and $\sigma\leq \frac{1}{31}m^{\frac{1}{4}}$ under the setting $d=r=1$, so our result is only worse by a constant factor.  

The upper bound $\sigma\leq \frac{1}{18}m^{\frac{1}{4}}$ in~\cite[Theorem 2.1]{Bandeira2017} is later improved to  $\sigma\leq O(\sqrt{\frac{m}{\log m}})$ in \cite{ZhongBoumal18PhaseSynchNearOptimalBounds}, based on a much more complicated argument and an  algorithmic implementation. {  After finishing this work, we became aware of a recent manuscript \cite{ling2021improved}, which investigates the synchronization-of-rotations problem using the method in \cite{ZhongBoumal18PhaseSynchNearOptimalBounds}, and proves that the original problem and the relaxed problem   have the same solution when $\sigma\leq O(\frac{\sqrt{m}}{d+\sqrt{d}\log m})$. While it is better than {our rate $\sigma\leq O(\frac{m^{1/2}}{d})$ when $r=d$}, our analysis investigates a more generic problem where $r$ could be smaller than $d$, and  establishes deterministic conditions that can be verified for a variety of probabilistic models. In comparison, the method in \cite{ling2021improved} is specifically designed for the additive Gaussian noise model. } 

{
While the results in this section are generalizations of the results in \cite{Bandeira2017} to the group of semi-orthogonal matrices, we remark that the generalization is nontrivial in two aspects. First, as commented in the conclusion of \cite{Bandeira2017}, the non-commutative nature of semi-orthogonal matrices renders the analysis more difficult. For example, the derivation in \eqref{eq:pertubation11} is more difficult than the corresponding equation in \cite[(4.3)]{Bandeira2017}. Second, we introduce a novel optimality certificate in \Cref{lemma:equavalency}, which is very different from the corresponding certificate in \cite[Lemma 4.4]{Bandeira2017}. In particular, our certificate concerns three variables: $c$, $\bT^{(1)}$, and $\bT^{(2)}$, while \cite[Lemma 4.4]{Bandeira2017} only depends on a single variable. More importantly, the certificate in \cite[Lemma 4.4]{Bandeira2017} has an explicit formula, but there is no explicit formula for the certificates $(c, \bT^{(1)},\bT^{(2)})$ in our work. To address this issue, we let $c=m/2$ and define $\bT^{(1)}$ and $\bT^{(2)}$ in a  constructive way in \eqref{eq:construction}.} 

{Ling \cite{ling2020solving} also proposes a generalization of \cite{Bandeira2017} to the group of orthogonal matrices, which can be considered as our setting with $r=d$. Similar to \cite[Lemma 4.4]{Bandeira2017}, the certificate in \cite[Proposition 5.1]{ling2020solving} is based on a single variable with an  explicit formula. While $-\bT^{(1)}$ in our work serves a similar purpose as the certificates in \cite[Lemma 4.4]{Bandeira2017} and \cite[Proposition 5.1]{ling2020solving}, $\bT^{(2)}$ and $c$ are required for our  setting and do not have explicit formula. In comparison, under the setting of orthogonal matrices (i.e., $r=d$), our rate is in the order of $\sigma=O(\frac{m^{1/4}}{d})$, which is slightly worse than the rate of $O(\frac{m^{1/4}}{d^{3/4}})$ in \cite{ling2020solving} by a factor of $d^{1/4}$. We suspect that our rate could be improved with a different way of constructing the certificates than  \eqref{eq:construction}, but we will leave it as a possible future direction.}

\section{Proof of main results}\label{sec:proof}
\subsection{Proof of \Cref{thm:main}}\label{sec:proof_thm:main}
Recall that  \eqref{eqn:tracemax} and \eqref{eq:tracesum_problem1} are equivalent in the sense that  $\tilde{\bU}_{ij}=\hat{\bO}_i\hat{\bO}_j^T$ for all $1\leq i,j\leq m$, 
{where $\tilde{\bU}=(\tilde{\bU}_{ij})$ is a solution to \eqref{eq:tracesum_problem1} and $\hat{\bO}=(\hat{\bO}_i)$ is a solution to \eqref{eqn:tracemax}.}
It is sufficient to show that \eqref{eq:tracesum_problem1} and its relaxation \eqref{eq:tracesum_relaxed} have the same solution. 
Then, the proof of \Cref{thm:main} can be divided into three  components as follows.

1. 
\Cref{lemma:equavalency} shows that if $\bS$ admits a decomposition $\bT^{(1)}+\bT^{(2)}+c\bI$ where $\bT^{(1)}$, $\bT^{(2)}$, and a solution of \eqref{eq:tracesum_problem1} satisfy the  conditions \eqref{eq:certificate1}-\eqref{eq:certificate2}, then this solution is also the unique solution to the relaxed problem  \eqref{eq:tracesum_relaxed}.

2. By constructing the certificates $\bT^{(1)}$ and $\bT^{(2)}$,  \Cref{lemma:noisy2} establishes \eqref{eq:noisy2_0}, a sufficient condition such that  \eqref{eq:certificate1}-\eqref{eq:certificate2} hold. 

3. \Cref{lemma:pertubation} establishes a perturbation result for the solution of \eqref{eq:tracesum_problem1}. When $\bW$ is small, the perturbation result can be used to verify \eqref{eq:noisy2_0}.
%
%


We first present our lemmas and a short proof of \Cref{thm:main} based on these lemmas, and leave the technical proofs of lemmas to Section~\ref{sec:proof_lemma}. 
\begin{lemma}[A condition for the equivalence between \eqref{eq:tracesum_problem1} and \eqref{eq:tracesum_relaxed}]\label{lemma:equavalency}Let $\tilde{\bU}$ be a solution to \eqref{eq:tracesum_problem1} and assume that it admits a decomposition $\tilde{\bU}=\tilde{\bV}\tilde{\bV}^T$ with $\tilde{\bV}\in\reals^{D\times r}$. 
If there exists a decomposition $\bS=\bT^{(1)}+\bT^{(2)}+c\bI$ such that \begin{align}\label{eq:certificate1}
\text{
{$\bT^{(1)} = \Proj_{\tilde{\bV}^\perp}\bT^{(1)}\Proj_{\tilde{\bV}^\perp}$}, \quad
$\bT^{(2)}_{ii}=\Proj_{\tilde{\bV}_i}\bT^{(2)}_{ii} \Proj_{\tilde{\bV}_i}$ for all $1\leq i\leq m$,}\\\label{eq:certificate2}
\text{  $\{P_{\tilde{\bV}_i}^T\bT^{(2)}_{ii} P_{\tilde{\bV}_i}\}_{i=1}^m$ and $-P_{\tilde{\bV}^\perp}^T\bT^{(1)}P_{\tilde{\bV}^\perp}$ are  positive definite matrices,}
\end{align} 
then $\tilde{\bU}$ is also the unique solution to the relaxed problem  \eqref{eq:tracesum_relaxed}. Therefore, \eqref{eq:tracesum_problem1} and \eqref{eq:tracesum_relaxed} have the same unique solution.
\end{lemma}

\begin{lemma}[A simplified condition in terms of the solution of \eqref{eq:tracesum_problem1}]\label{lemma:noisy2}
Let $\tilde{\bU}$ be a solution to \eqref{eq:tracesum_problem1} and assume that it admits a decomposition $\tilde{\bU}=\tilde{\bV}\tilde{\bV}^T$ with $\tilde{\bV}\in\reals^{D\times r}$. If
\begin{equation}\label{eq:noisy2_0}\frac{m}{2}
\geq \max_{1\leq i\leq m}\Big\|{\textstyle\sum_{j=1}^m\bW_{ij}\tilde{\bV}_i}\Big\|+2m\max_{1\leq i\leq m}\|\tilde{\bV}_i-\bTheta_i\|+\|{\bTheta}^T\tilde{\bV}-m\bI\|+\|\bW\|
,
\end{equation} then there exists  $\bT^{(1)}$ and $\bT^{(2)}$ such that $\bS=\bT^{(1)}+\bT^{(2)}+\frac{m}{2}\bI$, and \eqref{eq:certificate1}-\eqref{eq:certificate2} hold {with $c=m/2$}.
\end{lemma}

\begin{lemma}[Perturbation bounds of the solutions of  \eqref{eq:tracesum_problem1}]\label{lemma:pertubation}
 {If $m > \|\bW\|(4\sqrt{r}+1)+1$},
then for $\tilde{\bU}$, any solution to \eqref{eq:tracesum_problem1}, there is a  decomposition $\tilde{\bU}=\tilde{\bV}\tilde{\bV}^T$ with $\tilde{\bV}\in\reals^{D\times r}$ such that 
\begin{equation}\label{eq:lemma_pertubation1}
\begin{split}
&\|\tilde{\bV}-\bTheta\|_F\leq 4\|\bW\|
\textstyle
\sqrt{\frac{r}{m}}, 
\\
&\max_{1\leq i\leq m}\|[\bW\tilde{\bV}]_i\|_F\leq \max_{1\leq i\leq m}\|[\bW{\bTheta}]_i\|_F+4\|\bW\|^2
\textstyle
\sqrt{\frac{r}{m}}, 
\end{split}
\end{equation}
and
\begin{equation}\label{eq:lemma_pertubation2}
\max_{1\le i \le m}\|\tilde{\bV}_i-\bTheta_i\|_F\leq \frac{2\left(\max_{1\leq i\leq m}\|[\bW{\bTheta}]_i\|_F+4\|\bW\|^2\sqrt{\frac{r}{m}}\right)}
{ m-\|\bW\|(4\sqrt{r}+1)-1}
.
\end{equation}
\end{lemma}

\begin{proof}[Proof of \Cref{thm:main}]
\Cref{lemma:equavalency}  and \Cref{lemma:noisy2} imply that to prove \Cref{thm:main}, it is sufficient to prove \eqref{eq:noisy2_0}, which can be verified by applying  \Cref{lemma:pertubation}.
\end{proof}
\subsection{Proof of \Cref{cor:probablistic}}\label{sec:proof_cor:probablistic}
\begin{proof}[Proof of \Cref{cor:probablistic}]

Under the $\bTheta$-discordant property, inequality \eqref{eq:condition_deterministic} is satisfied if
$m$ is greater than
\[
\frac{8m[3\sigma\sqrt{drm\log m} + 36\sigma^2d\sqrt{rm}]}{m - 2 - 6\sigma\sqrt{dm}(2\sqrt{r}+1)}
+ 2[3\sigma\sqrt{drm\log m} + 36\sigma^2d\sqrt{rm}]
+ 24\sigma\sqrt{dr}
+ 6\sigma\sqrt{dm}
\]
or, by dividing the above expression by $m$,
\begin{align*}
    1 > \left( 2 + \frac{8}{1 - \frac{2}{m} - \frac{6\sigma\sqrt{d}(2\sqrt{r}+1)}{\sqrt{m}}}\right)
    \left[\frac{3\sigma\sqrt{dr\log m}}{\sqrt{m}} 
    + \frac{36\sigma^2d\sqrt{r}}{\sqrt{m}}\right]
    + \frac{24\sigma\sqrt{dr}}{m}
    + \frac{6\sigma\sqrt{d}}{\sqrt{m}}
    .
\end{align*}
If $\sigma \le \frac{m^{1/4}}{60\sqrt{dr}},$
then the RHS of the above inequality is upper-bounded by
\begin{align*}
    &\left( 2 + \frac{8}{1 - \frac{2}{m} - \frac{6}{60}\frac{2\sqrt{r}+1}{\sqrt{r}}\frac{1}{m^{1/4}}}\right)
    \left[\frac{3}{60}\frac{\sqrt{\log m}}{m^{1/4}} + \frac{36}{3600}\frac{1}{\sqrt{r}}\right]
    + \frac{24}{60}\frac{1}{m^{3/4}}
    + \frac{6}{60}\frac{1}{\sqrt{r}}\frac{1}{m^{1/4}}
    \\
    &\le
    \left( 2 + \frac{8}{1 - \frac{2}{m} - \frac{18}{60}\frac{1}{m^{1/4}}}\right)
     \left[\frac{3}{60}\frac{\sqrt{\log m}}{m^{1/4}} + \frac{36}{3600}\right]
    + \frac{24}{3600}\frac{1}{m^{3/4}}
    + \frac{6}{3600}\frac{1}{m^{1/4}}     
\end{align*}
since $r \ge 1$ and $\frac{2\sqrt{r}+1}{\sqrt{r}} \le 3$.
The last line is decreasing in $m$ if $m \ge 8$. 
At $m=8$, the denominator in the last line is $1 - \frac{2}{8} - \frac{18}{60}\frac{1}{8^{1/4}} > 0$ and the value of the whole line is less than $1$.
\end{proof}

\subsection{Proof of \Cref{cor:probablistic_gaussian}}\label{sec:proof_cor:probablistic_gaussian}
\begin{proof}[Proof of \Cref{cor:probablistic_gaussian}]
Considering \Cref{cor:probablistic}, it is sufficient to show that Gaussian noise $\bW$ satisfies the $\bTheta$-discordance with high probability under the MAXBET model.
Assume $\sigma^{-1}\bW_{ij}$ has i.i.d. standard normal entries. Then from $[\bW\bTheta]_i=\sum_{j=1}^m\bW_{ij}\bTheta_j\in\mathbb{R}^{d_i\times r}$, it is obvious that this matrix has zero-mean normal entries. To see the variance, note
$$
	\vec(\bW_{ij}\bTheta_j) = \vec(\bI_{d_i}\bW_{ij}\bTheta_j) 
	= (\bTheta_j^T\otimes\bI_{d_i})\vec(\bW_{ij})
	.
$$
Then $\Cov(\vec(\bW_{ij}))=\sigma^2\bI_{d_id_j}$ and
\begin{align*}
	\Cov(\vec(\bW_{ij}\bTheta_j))
	&= \sigma^2(\bTheta_j^T\otimes\bI_{d_i})(\bTheta_j^T\otimes\bI_{d_i})^T
	= \sigma^2(\bTheta_j^T\otimes\bI_{d_i})(\bTheta_j\otimes\bI_{d_i})
	\\
	&= \sigma^2(\bTheta_j^T\bTheta_j\otimes\bI_{d_i}\bI_{d_i})
	= \sigma^2(\bI_r\otimes\bI_{d_i})
	= \sigma^2\bI_{rd_i}
	,
\end{align*}
i.e., $\bW_{ij}\bTheta_j$ has i.i.d. normal entries with variance $\sigma^2$.
Then $[\bW\bTheta]_i$ is the sum of $m$ i.i.d. copies of $\bW_{ij}\bTheta_j$, hence entries have variance $m\sigma^2$.
Now from Theorem 9.26 of \cite{Foucart2013},
$$
	\Pr{\left(\frac{1}{\sigma\sqrt{m}}\|[\bW\bTheta]_i\| \ge \sqrt{d_i} + \sqrt{r} + t\right)} \le e^{-t^2/2}
$$
for $t \ge 0$.
Applying the union bound and noting that $\frac{1}{\sqrt{r}}\|[\bW\bTheta]_i\|_F \le \|(\bW\bTheta)_i\|_2$, we obtain
$$
	\Pr{\left(\max_{i=1,\dotsc,m}\|[\bW\bTheta]_i\|_F \le \sigma(\sqrt{\underline{d}rm} + r\sqrt{m} + t\sqrt{r})\right)} > 1 - me^{-t^2/2}
	,
$$
where $\underline{d}=\min_{i=1,\dotsc,m} d_i$.
Now choose $t$ such that $e^{-t^2/2}=1/m^2$, i.e., $t=2\sqrt{\log m}$. Then, 
\begin{equation}\label{eqn:concentrationlocal}
	\Pr{\left(\max_{i=1,\dotsc,m}\|[\bW\bTheta]_i\|_F \le \sigma(\sqrt{\underline{d}rm} + r\sqrt{m} + 2\sqrt{r\log m})\right)} > 1 - \frac{1}{m}
	.
\end{equation}
Since $\underline{d}\ge \max\{r, 2\}$ and $m \ge 2$, we have $r \le \sqrt{\underline{d}r}$ and $\sqrt{\underline{d}m} \ge 2$. 
Furthermore, if $m \ge 3$, then $m \le m\log m$.
Thus
\begin{equation}\label{eqn:concentration_concordance}
 \sqrt{\underline{d}rm} + r\sqrt{m} + 2\sqrt{r\log m}
 \le
 3\sqrt{\underline{d}rm\log m}
 \le
 3\sqrt{Dr\log m}
 .
\end{equation}
If $m=2$ and $\underline{d} \ge 6$,
\begin{align*}
\sqrt{2\underline{d}r} + \sqrt{2r^2} + 2\sqrt{r\log 2}
&\leq
3\sqrt{2\underline{d}r\log 2}
\leq
3\sqrt{Dr\log 2}
.
\end{align*}
Thus if  $m\ge 3$ or $m \ge 2$ and $\underline{d} \ge 6$, then
$
	\max_{i=1,\dotsc,m}\|[\textstyle\frac{1}{\sigma}\bW\bTheta]_i\|_F \le 3\sqrt{Dr\log m}
$
with probability at least $1 - 1/m$.

To bound $\|\bW\|$, 
{
observe that $\bW \stackrel{d}{=} \bW^{(1)} + \bW^{(2)}$, where $\bW^{(1)}\in\mathbb{R}^{D\times D}$ has entries i.i.d. from $N(0, \sigma^2/2)$, 
and $\bW^{(2)}$ is generated as follows: $[\bW^{(2)}]_{ij} = [\bW^{(1)}]_{ji}^T$ for $i \neq j$, and
$[\bW^{(2)}]_{ii}$ has entries i.i.d. from $N(0, \sigma^2/2)$ under \eqref{eq:model}, or $[\bW^{(2)}]_{ii} = -[\bW^{(2)}]_{ii}$ under \eqref{eq:model2}.
Marginally both $\bW^{(1)}$ and $\bW^{(2)}$ have entries i.i.d. from $N(0, \sigma^2/2)$.
Then, \cite[Theorem II.13]{Davidson2001} implies that
\begin{equation*}
    \Pr{\left(\frac{\sqrt{2}}{\sigma}\Vert\bW^{(1)}\Vert \ge 2\sqrt{D} + t\right)}
    =
    \Pr{\left(\frac{\sqrt{2}}{\sigma}\Vert\bW^{(2)}\Vert \ge 2\sqrt{D} + t\right)}
    < e^{-t^2/2}
    .
\end{equation*}
Applying the union bound and $\Vert \bW \Vert \le \Vert \bW^{(1)}\Vert + \Vert \bW^{(2)} \Vert$ yields %
}
\begin{equation*}
\Pr{\left(\|\bW\| \le \sigma\sqrt{2}(2\sqrt{D}+t)\right)} > 1 - 2e^{-t^2/2}
\end{equation*}
for $t \ge 0$. Choose $t=(\frac{3}{\sqrt{2}}-2)\sqrt{D}$ to have
$
	\Pr{\left(\|\bW\| \le 3\sigma\sqrt{D}\right)} > 1 - 2e^{-\frac{(3-2\sqrt{2})^2}{4}D}
	.
$
\end{proof}
\subsection{Proof of \Cref{cor:consistency1}}
The proof follows from \eqref{eq:lemma_pertubation2} in \Cref{lemma:pertubation} and \Cref{cor:probablistic}.

\subsection{Proof of \Cref{thm:main2}}\label{sec:proofmain2:thm}
As a preparation, we provide intermediate results first. Proofs of these results are provided in Section \ref{sec:proofs}.  
\begin{lemma}\label{lem:weyllocal}
    Let $\bLambda_i$ be the Lagrange multiplier of a 
    critical point $(\bO_1, \dotsc, \bO_m)$ of problem \eqref{eqn:tracemax}. That is, it is a symmetric $r\times r$ matrix satisfying $\bO_i\bLambda_i=\sum_{j=1}^m\bS_{ij}\bO_j$.
    Then,
    for block matrices $\bO=[\bO_1^T, \dotsc, \bO_m^T]^T$ and 
    $\bTheta=[\bTheta_1^T, \dotsc, \bTheta_m^T]^T$,
    the following holds
    under \eqref{eq:model}:
    \begin{align*}
    \|\bLambda_i - m\bI\| \le
    \textstyle
    \|\sum_{j=1}^m\bW_{ij}\bO_j\|
    + m\|\bO_i^T\bTheta_i - \bI_r\|
    + \|\bTheta^T\bO - m\bI_r\|
    .
    \end{align*}
    Under \eqref{eq:model2}, we have
    \begin{align*}
    \|\bLambda_i - (m-1)\bI\| \le
    \textstyle
    \|\sum_{j\neq i}\bW_{ij}\bO_j\|
    + m\|\bO_i^T\bTheta_i - \bI_r\|
    + \|\bTheta^T\bO - m\bI_r\|
    .
    \end{align*}
\end{lemma}

Results parallel to \Cref{lemma:pertubation} are also obtained: 
\begin{lemma}\label{lem:upperboundlocal}
    Let $(\bO_1, \dotsc, \bO_m)\in\times_{i=1}^m\mathcal{O}_{d,r}$ be a
    {candidate}
    critical point of problem \eqref{eqn:tracemax},
    {i.e., a point}
	satisfying Assumption \ref{assumption}. 
	If we build a block matrix $\bO=[\bO_1^T, \dotsc, \bO_m^T]^T$, then the following error estimates hold.
	\begin{align}
	    \|\bO - \bTheta\|_F &\le
	    \begin{cases}
	        4\|\bW\|\frac{\sqrt{r}}{\sqrt{m}},
	        & \text{under \eqref{eq:model}}, 
	        \\
	        4\|\bW\|\frac{\sqrt{r}}{\sqrt{m}-2/\sqrt{m}},
	        & \text{under \eqref{eq:model2}},
	    \end{cases}
	    \label{eqn:localupperbound1}
	    \\
	    \|\bTheta^T\bO - m\bI_r\| &\le
	    \begin{cases}
	        4\|\bW\|\sqrt{r},
	        & \text{under \eqref{eq:model}}, 
	        \\
	        4\|\bW\|\frac{\sqrt{r}}{1-2/m},
	        & \text{under \eqref{eq:model2}},
	    \end{cases}
	    \label{eqn:localupperbound1_1}
	    \\	    
	    \max_{1\le i\le m}\|[\bW\bO]_i]\|_F &\le
	    \max_{1\le i\le m}\|[\bW\bTheta]_i\|_F 
	    +
	    \begin{cases}
	        4\|\bW\|^2\frac{\sqrt{r}}{\sqrt{m}}, 
	        & \text{under \eqref{eq:model}}, 
	        \\
	        4\|\bW\|^2\frac{\sqrt{r}}{\sqrt{m}-2/\sqrt{m}},
	        & \text{under \eqref{eq:model2}},
	    \end{cases}
	    \label{eqn:localupperbound2}
	    \\
	    \max_{1\le i\le m}\|\bO_i - \bTheta_i\|_F &\le
	    \begin{cases}
	        \frac{2\left(\max_{1\le i\le m}\|[\bW\bTheta]_i\|_F + 4\|\bW\|^2\frac{\sqrt{r}}{\sqrt{m}}\right)}{m - 4\|\bW\|\sqrt{r}},
	        & \text{under \eqref{eq:model}},
	        \\
	        \frac{2\left(\max_{1\le i\le m}\|[\bW\bTheta]_i\|_F + 4\|\bW\|^2\frac{\sqrt{r}}{\sqrt{m}-2/\sqrt{m}}\right)}{m - 4\|\bW\|\frac{\sqrt{r}}{1-2/m}-3},
	        & \text{under \eqref{eq:model2}}
	   \end{cases}
	   \label{eqn:localupperbound3}
	\end{align}
	where
	\begin{align*}
	    m > 
	    \begin{cases}
	        4\|\bW\|\sqrt{r},
	        & \text{under \eqref{eq:model}},
	        \\
	        \frac{5+4\|\bW\|\sqrt{r} + \sqrt{16\|\bW\|^2 r + 40 \|\bW\|\sqrt{r} + 1}}{2},
	        & \text{under \eqref{eq:model2}}.
	    \end{cases}
	\end{align*}
\end{lemma}

We now assume 
the data model \eqref{eq:model}.
We want a condition on the noise matrices $\bW_{ij}$ that guarantees the certificate \eqref{eqn:certificate} to hold.
Since $\bL\bO=\bzero$ if $\bO=(\bO_1,\dotsc,\bO_m)$ is a first-order critical point, it suffices to find a condition that
$$
	\bx^T\bL\bx \ge 0 \quad \text{for all $\bx=(\bx_1,\dotsc,\bx_m), \bx_i \in \mathbb{R}^{d}$ such that $\bO^T\bx =\bzero$}.
$$
Then for $\bO$ being a 
{candidate}
critical point and $\bx$ satisfying $\bO^T\bx=\bzero$,
\begin{align*}
	\bx^T\bL\bx &=
	\sum_{i=1}^m\left(
		\bx_i^T\bO_i\bLambda_i\bO_i^T\bx_i + \tau_i\bx_i^T\bO_i^{\perp}\bO_i^{\perp T}\bx
	\right)
	- \bx^T\bS\bx
	\\
	&\ge
	\sum_{i=1}^m\left(
		\tau_i\bx_i^T\bO_i\bO_i^T\bx_i + \tau_i\bx_i^T\bO_i^{\perp}\bO_i^{\perp T}\bx_i
	\right)
	- \bx^T\bS\bx=
	\sum_{i=1}^m\tau_i\|\bx_i\|^2 - \bx^T\bS\bx
	.
\end{align*}

The block matrix \eqref{eqn:blockS} can be written 
\begin{equation}\label{eqn:Sexpansion}
	\bS = \bTheta\bTheta^T 
	+ \bW,
\end{equation}
where 
$\bW$ is a block matrix constructed from $\bW_{ij}$ in a similar fashion to \eqref{eqn:blockS}.
Then
\begin{align}\label{eqn:quadraticform}
	\bx^T\bS\bx &= \bx^T\bTheta\bTheta^T\bx 
	+ \bx^T\bW\bx=
	\bx^T(\bTheta-\bO)(\bTheta-\bO)^T\bx 
	+ \bx^T\bW\bx
	\nonumber\\
	&= \|(\bTheta-\bO)^T\bx\|^2 
	+ \bx^T\bW\bx\le
	\|\bTheta-\bO\|^2\|\bx\|^2 
	+ \|\bW\|\|\bx\|^2
	.
\end{align}
The second equality is due to
$\bO^T\bx=\bzero$. Hence we have
\begin{equation}\label{eqn:lowerbound1}
	\bx^T\bL\bx \ge
	\sum_{i=1}^m\tau_i\|\bx_i\|^2 - \|\bTheta-\bO\|^2\|\bx\|^2 - \|\bW\|\|\bx\|^2
	.
\end{equation}
Combining Weyl's inequality and \Cref{lem:weyllocal},
we obtain a lower bound on $\tau_i$:
$$
	\tau_i \ge 
	m
	- \|[\bW\bO]_i\| - m\|\bO_i^T\bTheta_i-\bI\| - \|\bTheta^T\bO - m\bI\|
	.
$$
Substituting this with inequality \eqref{eqn:lowerbound2}, we see
\begin{equation}\label{eqn:lowerbound2}
\begin{split}
	\bx^T\bL\bx &\ge
	( m - \|\bTheta^T\bO - m\bI\|)\|\bx\|^2
	- \|\bTheta - \bO\|^2\|\bx\|^2 
	\\
	& \quad\quad 
		- \sum_{i=1}^m\left(
		\|[\bW\bO]_i\|\|\bx_i\|^2 + m\|\bO_i^T\bTheta_i-\bI\|\|\bx_i\|^2
	\right)
	- \|\bW\|\|\bx\|^2
	\\
	&\ge
	\left(
	m - \|\bTheta^T\bO - m\bI\| 
	- \|\bTheta - \bO\|^2 
	\right.
	\\
	& \quad\quad
	\left.
	- \max_{i=1,\dotsc,m}\|[\bW\bO]_i\| 
	- m\max_{i=1,\dotsc,m}\|\bO_i^T\bTheta_i - \bI\| 
	- \|\bW\|
	\right)
	\|\bx\|^2
	.
\end{split}
\end{equation}
Thus if
\begin{equation}\label{eqn:lowerbound3}
	\begin{split}
	m 
	& \ge
	\|\bTheta^T\bO - m\bI\| 
	+ \|\bTheta - \bO\|^2 
	\\
	&\qquad
	+ \max_{i=1,\dotsc,m}\|[\bW\bO]_i\| 
	+ m\max_{i=1,\dotsc,m}\|\bO_i^T\bTheta_i - \bI\| 
	+ \|\bW\|
	,
	\end{split}
\end{equation}
then we have $\bL \psdge \bzero$.

Also, since post-multiplying a fully orthogonal matrix in $\mathcal{O}_{r,r}$ to $\bO_1, \dotsc, \bO_m$ does not change the objective of \eqref{eqn:tracemax}, 
we may choose $\bO$ such that $\bTheta^T\bO$ is symmetric, positive semidefinite.

Thus we apply \Cref{lem:upperboundlocal}  to the right-hand side of inequality \eqref{eqn:lowerbound3} to get:
\begin{align*}
	&\|\bTheta^T\bO - m\bI\| + \max_{i=1,\dotsc,m}\|(\bW\bO)_i\| 
	\quad
	+ m\max_{i=1,\dotsc,m}\|\bO_i^T\bTheta_i - \bI\| 
	+ \|\bTheta - \bO\|^2 
	+ \|\bW\|
	\\
	&\le
    4\|\bW\|\sqrt{r}
    + \max_{1\le i\le m}\|[\bW\bTheta]_i\|_F + 4\|\bW\|^2 
    \textstyle
    \sqrt{\frac{r}{m}}
    \\
    &\qquad
    + \frac{2m(\max_{1\le i\le m}\|[\bW\bTheta]_i\|_F + 4\|\bW\|^2\sqrt{\frac{r}{m}})}{m - 4\|\bW\|\sqrt{r}}
    + 16\|\bW\|^2\frac{r}{m}
    + \|\bW\|
	.
\end{align*}
If this bound is less than or equal to $m$, then condition \eqref{eqn:lowerbound3} is satisfied. Thus we have proved Theorem \ref{thm:main2}.

\subsection{Proof of \Cref{cor:discordant}}\label{sec:proof_discordant}
Under the $\bTheta$-concordance of $\frac{1}{\sigma}\bW$,
the right-hand side of inequality  \eqref{eqn:noiselevel} in \Cref{thm:main2} is upper bounded by
\begin{equation}\label{eqn:concordanceupperbound1}
\begin{split}
12\sigma\sqrt{Dr} 
&+ 3\sigma\sqrt{Dr\log m}
+ 36\sigma^2 D \sqrt{\frac{r}{m}}
+ \frac{2m(3\sigma\sqrt{Dr\log m} + 36\sigma^2 D \sqrt{\frac{r}{m}})}{m - 12\sigma\sqrt{Dr}}
\\
&+ 144\sigma^2\frac{Dr}{m}
+ 3\sigma\sqrt{D}
\end{split}
\end{equation}
if $\sigma < \frac{m}{12\sqrt{Dr}}$.
If \eqref{eqn:concordanceupperbound1} is less than or equal to $m$, or equivalently
\begin{equation}\label{eqn:concordanceupperbound2}
\begin{split}
    1 &\geq
    12\sigma\sqrt{\frac{dr}{m}}
    + 3\sigma\sqrt{\frac{dr\log m}{m}}
    + 36\sigma^2 d\sqrt{\frac{r}{m}}
    \\&\qquad 
    + \frac{2(3\sigma\sqrt{drm\log m} + 36\sigma^2 d\sqrt{rm})}{m - 12\sigma\sqrt{drm}}
    + 144\sigma^2\frac{dr}{m}
    + 3\sigma\sqrt{\frac{d}{m}}
\end{split}    
\end{equation}
for $
    \sigma < \frac{m^{1/2}}{12\sqrt{dr}}
    ,$
then from the proof of \Cref{thm:main2}, we see that condition \eqref{eqn:lowerbound3} is satisfied and thus the claim is proved.

The fourth term in the right-hand side of inequality \eqref{eqn:concordanceupperbound2} is 
\begin{align*}
    \frac{2\left(3\sigma\sqrt{\frac{dr\log m}{m}} + 36\sigma^2 d\sqrt{\frac{r}{m}}\right)}{1 - 12\sigma\sqrt{\frac{dr}{m}}}
    \leq
    \frac{2}{1 - \frac{12}{31}\frac{1}{m^{1/4}}}
    \left(3\sigma\sqrt{\frac{dr\log m}{m}} + 36\sigma^2 d\sqrt{\frac{r}{m}}\right)
\end{align*}
if $\sigma \le \frac{m^{1/4}}{31\sqrt{dr}}$.
Thus, by replacing $\sigma$ with $\frac{m^{1/4}}{31\sqrt{dr}}$, the RHS of \eqref{eqn:concordanceupperbound2} is upper bounded by
\begin{align*}
    \frac{12}{31}\frac{1}{m^{1/4}}
    + 
    \left(1 +
    \frac{2}{1 - \frac{12}{31}\frac{1}{m^{1/4}}}\right)
    \left(
    \frac{3}{31}\frac{\sqrt{\log m}}{m^{1/4}}
    + \frac{36}{961}\frac{1}{\sqrt{r}}
    \right)
    + \frac{144}{961}\frac{1}{m^{1/2}}
    + \frac{3}{31}\frac{1}{\sqrt{r}m^{1/4}}
    .
\end{align*}
Since $r \ge 1$, $\frac{\sqrt{\log m}}{m^{1/4}} \le \sqrt{\frac{2}{e}}$, and the rest of the terms are decreasing in $m$, the above quantity is less than $1$ for $m\ge 2$.

\subsection{\Cref{thm:main2}, \Cref{cor:discordant},  \Cref{cor:gaussian} under \eqref{eq:model2}}\label{sec:proof:maxdiff}
Under the MAXDIFF model, inequality \eqref{eqn:quadraticform} is replaced by
\begin{align*}
	\bx^T\bS\bx 
	&\leq
	\|\bTheta-\bO\|^2\|\bx\|^2 
	- \sum_{i=1}^m\|\bTheta_i\bTheta_i^T\bx_i\|^2 
	+ \|\bW\|\|\bx\|^2
	\\
	&\leq
	\Big(\|\bTheta-\bO\|^2
	- \min_{1\le i \le m}\|\bTheta_i\|^2
	+ \|\bW\|\Big)
	\|\bx\|^2	
	,
\end{align*}
and \eqref{eqn:lowerbound2} by
\begin{align*}
	\bx^T\bL\bx &\ge
	\Big(
	m - 1 
	- \|\bTheta^T\bO - m\bI\| - \max_{i=1,\dotsc,m}\|[\bW\bO]_i\| 
	\\
	& \quad\quad
	- m\max_{i=1,\dotsc,m}\|\bO_i^T\bTheta_i - \bI\| - \|\bTheta - \bO\|^2 - \|\bW\|
	+ \min_{1\le i \le m}\|\bTheta_i\|^2
	\Big)
	\|\bx\|^2
	\\
	&=
	\Big(
	m 
	- \|\bTheta^T\bO - m\bI\| - \max_{i=1,\dotsc,m}\|[\bW\bO]_i\| 
	\\
	& \quad\quad
	- m\max_{i=1,\dotsc,m}\|\bO_i^T\bTheta_i - \bI\| - \|\bTheta - \bO\|^2 - \|\bW\|
	\Big)
	\|\bx\|^2
	.
\end{align*}
since $\|\bTheta_i\|=1$ for all $i$.
Thus condition \eqref{eqn:lowerbound3} for $\bL \psdge \bzero$ to hold remains unchanged.
Applying \Cref{lem:upperboundlocal}, we obtain
\begin{align*}
    m
	&\ge
	\textstyle
    4\|\bW\|\frac{\sqrt{r}}{1-2/m}
    + \max_{1\le i\le m}\|[\bW\bTheta]_i\|_F + 4\|\bW\|^2\frac{\sqrt{r}}{\sqrt{m}-2/\sqrt{m}}
    + \|\bW\|
    \\
    & \quad
    + \frac{2m(\max_{1\le i\le m}\|[\bW\bTheta]_i\|_F + 4\|\bW\|^2\frac{\sqrt{r}}{\sqrt{m}-2/\sqrt{m}})}{m - 4\|\bW\|\frac{\sqrt{r}}{1-2/m}-3}
    + 
    \textstyle
    16\|\bW\|^2\frac{r}{(\sqrt{m}-2/\sqrt{m})^2}
	.
\end{align*}
Proceeding as above for MAXBET, we obtain the bound on $\sigma$ as stated in \Cref{rem:maxdifflocal}.

Furthermore, inequality \eqref{eqn:concentrationlocal} is replaced by
\[
	\Pr{\left(\max_{i=1,\dotsc,m}\|[\bW\bTheta]_i\|_F \le \sigma(\sqrt{\underline{d}r(m-1)} + r\sqrt{m-1} + 2\sqrt{r\log m})\right)} > 1 - \frac{1}{m}
	,
\]
(recall that $\underline{d}=\min_{m=1,\dotsc,m}d_i$)
and inequality \eqref{eqn:concentration_concordance} holds for $m \ge 2$ for all $\underline{d}$, since $m - 1 \le m\log m$ for all $m \ge 2$.
Thus the conclusion of \Cref{cor:gaussian} holds without modification, provided that $m \ge 4$ and $\sigma \le \frac{m^{1/4}}{64\sqrt{dr}}$ as stated in \Cref{rem:maxdifflocal}.

{%
\subsection{Proof of \Cref{prop:consistencylocal}}
The desired results follow immediately from inequality \eqref{eqn:localupperbound3} of \Cref{lem:upperboundlocal} and \Cref{def:discordance}.%
}

\section{Proofs of Technical Lemmas and Propositions}\label{sec:proof_lemma}

\subsection{Proof of \Cref{prop:noiseless}}\label{sec:proof_prop:noiseless}
\begin{proof}[Proof of \Cref{prop:noiseless}]
    First consider model \eqref{eq:model}. 
	We have $\bS_{ij}\allowbreak=\allowbreak\bTheta_i\bTheta_j^T$ for all $i, j$. Then the objective of \eqref{eqn:tracemax} is
	$$
		\sum_{i,j}\tr(\bO_i^T\bTheta_i\bTheta_j^T\bO_j)
		= \sum_{i,j}\tr[(\bTheta_i^T\bO_i)^T(\bTheta_j^T\bO_j)]
		.
	$$
	Each term is bounded by the von Neumann-Fan inequality \cite[Example 2.8.7]{lange2016mm}
	\begin{equation}\label{eqn:vonneumann}
		\tr[(\bTheta_i^T\bO_i)^T(\bTheta_j^T\bO_j)]
		\le
		\sum_{k=1}^r\sigma_k(\bTheta_i^T\bO_i)\sigma_k(\bTheta_j^T\bO_j)
		,
	\end{equation}
	where $\sigma_k(\bM)$ is the $k$th largest singular value of matrix $\bM$.
	Since $\bO_i^T\bTheta_i\bTheta_i^T\bO_i \psdle \bO_i^T\bO_i = \bI_r$, we see $\max_{k=1,\cdots,r}\sigma_k(\bTheta_i^T\bO_i) \le 1$ for all $\bO_i \in \mathcal{O}_{d_i,r}$, $i=1,\dotsc,m$.
	Thus the largest possible value of the right-hand side of inequality \eqref{eqn:vonneumann} is $r$ and \eqref{eqn:tracemax} has maximum $m^2r$.
	This is achieved by $\bO_i=\bTheta_i$ for $i=1,\dotsc,m$ since $\bTheta_i^T\bTheta_i=\bI_r$.
	
	It is straightforward to modify the above proof for model \eqref{eq:model2}. The maximum is $m(m-1)r$.
\end{proof}

\subsection{Proofs of Lemmas for  \Cref{thm:main}}\label{sec:proof_lemma:local}

%
%

\begin{proof}[Proof of \Cref{lemma:equavalency}]
For any $\bU$ in the constraint set of  \eqref{eq:tracesum_relaxed} such that $\bU\neq \tilde{\bU}$, and $\bX=\bU-\tilde{\bU}$, we have $P_{\tilde{\bV}^\perp}^T\bX P_{\tilde{\bV}^\perp}=P_{\tilde{\bV}^\perp}^T\bU P_{\tilde{\bV}^\perp}-P_{\tilde{\bV}^\perp}^T\tilde{\bU} P_{\tilde{\bV}^\perp}= P_{\tilde{\bV}^\perp}^T\bU P_{\tilde{\bV}^\perp}\psdge \mathbf{0}$, and $P_{\tilde{\bV}_i}^T\bX_{ii}P_{\tilde{\bV}_i}=P_{\tilde{\bV}_i}^T\bU_{ii}P_{\tilde{\bV}_i}-P_{\tilde{\bV}_i}^T\tilde{\bU}_{ii}P_{\tilde{\bV}_i}=P_{\tilde{\bV}_i}^T\bU_{ii}P_{\tilde{\bV}_i}-\bI\psdle \mathbf{0}$. In summary, $\bX$ has the properties of
\begin{equation}\label{eq:tangent4}
P_{\tilde{\bV}^\perp}^T\bX P_{\tilde{\bV}^\perp}\psdge \mathbf{0},~\text{$\tr(\bX_{ii})=0$, and  $P_{\tilde{\bV}_i}^T\bX_{ii}P_{\tilde{\bV}_i}\psdle \mathbf{0}$ for all $1\leq i\leq m$.} 
\end{equation}

In addition, either $P_{\tilde{\bV}^\perp}^T\bX P_{\tilde{\bV}^\perp}$ is nonzero or  $P_{\tilde{\bV}_i}^T\bX_{ii} P_{\tilde{\bV}_i}$ is nonzero {for some $i$}. If they are all zero matrices, then we have
\begin{align}\label{eq:tangent5}
P_{\tilde{\bV}^\perp}^T\bU P_{\tilde{\bV}^\perp}=P_{\tilde{\bV}^\perp}^T\tilde{\bU} P_{\tilde{\bV}^\perp}=\mathbf{0},\\
\label{eq:tangent6}
P_{\tilde{\bV}_i}^T\bU_{ii} P_{\tilde{\bV}_i}=P_{\tilde{\bV}_i}^T\tilde{\bU}_{ii} P_{\tilde{\bV}_i}=\bI_r.
\end{align}
Since  $\bU_{ii}\psdge 0$, we have {$\tilde{\bV}_i^T\bU_{ii}\tilde{\bV}_i\psdge \mathbf{0}$}. Combining it with $\tr(P_{\tilde{\bV}_i}^T\bU_{ii} P_{\tilde{\bV}_i})=r$ {(due to equation \eqref{eq:tangent6})} and $r=\tr(\bU_{ii})=\tr(P_{\tilde{\bV}_i}^T\bU_{ii} P_{\tilde{\bV}_i})+\tr(\tilde{\bV}_i^T\bU_{ii}\tilde{\bV}_i)$, we have {$\tilde{\bV}_i^T\bU_{ii}\tilde{\bV}_i= \mathbf{0}$}. Combining it with $\bU_{ii}\psdge \mathbf{0}$, we have {$\tilde{\bV}_i^T\bU_{ii}=\mathbf{0}$ and $\bU_{ii}^T\tilde{\bV}_i=\mathbf{0}$}. It implies that 
{
$\bU_{ii}=\tilde{V}_i Z_i \tilde{V}_i^T$ for some positive semidefinite $Z_i$.
That $\bU_{ii} \psdle \bI$ and $\tr(\bU_{ii})=r$ in turn implies that $Z_i = \bI_r$. Thus,
\begin{equation}\label{eq:tangent7}
\bU_{ii}=\tilde{\bV}_i\tilde{\bV}_i^T.
\end{equation}%
}
In addition, \eqref{eq:tangent5} and $\bU\psdge \mathbf{0}$ mean that  $\bU=\Proj_{\tilde{\bV}}^T\bU \Proj_{\tilde{\bV}}$, that is, there exists a matrix $\bZ\in\mathbb{R}^{r\times r}$ such that $\bU=\tilde{\bV}\bZ\tilde{\bV}^T$, and as a result, $\bU_{ii}=\tilde{\bV}_i\bZ\tilde{\bV}_i^T$. Combining it with \eqref{eq:tangent7}, we have $\bZ=\bI$ and $\bU=\tilde{\bV}\tilde{\bV}^T=\tilde{\bU}$, which is a contradiction to {assumption} $\bU\neq \tilde{\bU}$.

Combining the property of $\bX$ in \eqref{eq:tangent4} with the assumption of $\bT$ in  \eqref{eq:certificate2} that $\{P_{\tilde{\bV}_i}^T\bT^{(2)}_{ii} P_{\tilde{\bV}_i} \}_{i=1}^m$ and $-P_{\tilde{\bV}^\perp}^T\bT^{(1)}P_{\tilde{\bV}^\perp}$ are  positive definite matrices, we have
\begin{equation}\label{eq:tangent8}
\begin{split}
\tr(\bX\bS)
&=
\tr(\bX\bT^{(1)})+\tr(\bX\bT^{(2)})+c\tr(\bX)
\\
&= 
\textstyle
\tr[(P_{\tilde{\bV}^\perp}^T\bX P_{\tilde{\bV}^\perp})( P_{\tilde{\bV}^\perp}^T\bT^{(1)}P_{\tilde{\bV}^\perp})]+\sum_{i=1}^m\tr(\bX_{ii} \bT^{(2)}_{ii})
\\
&=
\tr[(P_{\tilde{\bV}^\perp}^T\bX P_{\tilde{\bV}^\perp})( P_{\tilde{\bV}^\perp}^T\bT^{(1)}P_{\tilde{\bV}^\perp})] 
\\
&\qquad 
\textstyle
+ \sum_{i=1}^m\tr[(P_{\tilde{\bV}_i}^T\bX_{ii} P_{\tilde{\bV}_i})(P_{\tilde{\bV}_i}^T\bT^{(2)}_{ii} P_{\tilde{\bV}_i})]
< 0.
\end{split}
\end{equation}
{The first equality uses assumption \eqref{eq:certificate1}.}
The last inequality is strict because either $P_{\tilde{\bV}^\perp}^T\bX P_{\tilde{\bV}^\perp}$ is nonzero or  $P_{\tilde{\bV}_i}^T\bX_{ii} P_{\tilde{\bV}_i}$ is nonzero for some $1\leq i\leq m$. 
{Then equation \eqref{eq:tangent8}}
implies that $\trace(\bS\bU)< \trace(\bS\tilde{\bU})$ for all $\bU\neq\tilde{\bU}$, and as a result, $\tilde{\bU}$ is the unique solution to  \eqref{eq:tracesum_relaxed}.
\end{proof}

\begin{proof}[Proof of \Cref{lemma:noisy2}]
In this proof, we aim to construct the certificate in \Cref{lemma:equavalency}. The process can be divided into three steps:
\begin{itemize}
    \item Find a decomposition of $\bS=\bS^{(1)}+\bS^{(2)}$ based on the first-order optimality.
    \item Construct the certificate $\bT^{(1)}$ and $\bT^{(2)}$ from the decomposition $\bS^{(1)}$ and $\bS^{(2)}$. The explicit expression is given in \eqref{eq:construction}.
    \item Verify that the 
{certificate satisfies}
    the conditions in \Cref{lemma:equavalency}. 
\end{itemize}
\textbf{Step 1: decomposition of $\bS$ based on the first-order optimality.}
We investigate the first order condition for any  solution of \eqref{eq:tracesum_problem1} and summarize the result in Lemma~\ref{lemma:stationary} as below:
\begin{lemma}\label{lemma:stationary}
{Let $\tilde{\bU}=\tilde{\bV}\tilde{\bV}^T$ be a solution to \eqref{eq:tracesum_problem1} with $\tilde{\bV}\in\reals^{D\times r}$. Then the input matrix $\bS$ can be decomposed into $\bS=\bS^{(1)}+\bS^{(2)}$, where $\bS^{(1)}$ and $\bS^{(2)}$ are such that}
\begin{align}\label{eq:tildeT1}
[\bS^{(1)}]_{ij}
&=
\begin{cases}
\bS_{ij}, & i\neq j, \\
\bS_{ii}-\sum_{j=1}^m\bS_{ij}\tilde{\bV}_j\tilde{\bV}_i^T, & i = j,
\end{cases}\\
\label{eq:tildeT2}
[\bS^{(2)}]_{ij}
&=
\begin{cases}
\mathbf{0}, & i\neq j, \\ \sum_{j=1}^m\bS_{ij}\tilde{\bV}_j\tilde{\bV}_i^T,\,\,\,\,\,\,\,\,\,\,\,\,\,\,\, & i = j,
\end{cases}
\end{align}
{
and satisfy that
\begin{equation}\label{eq:KKT}
\bS^{(1)}  = \Proj_{\tilde{\bV}^\perp}{\bS}^{(1)}\Proj_{\tilde{\bV}^\perp}
\quad\text{and}\quad
\bS^{(2)}_{ii}  = \Proj_{\tilde{\bV}_i}\bS^{(2)}_{ii} \Proj_{\tilde{\bV}_i}
\quad\text{for all}~ 
1\leq i\leq m.
\end{equation}%
}%
\end{lemma}
The properties of $\bS^{(1)}$ and $\bS^{(2)}$ in \eqref{eq:KKT} is exactly the same as the  condition \eqref{eq:certificate1} for certificates $\bT^{(1)}$ and $\bT^{(2)}$ in \Cref{lemma:equavalency}. As a result, it is convenient to construct our certificates $\bT^{(1)}$ and $\bT^{(2)}$ based on $\bS^{(1)}$ and $\bS^{(2)}$. In fact, the explicit expression of \eqref{eq:construction} in step 2 shows that  $\bT^{(1)}$ is derived from $\bS^{(1)}$ and $\bT^{(2)}$ is derived from  $\bS^{(2)}$.

\begin{proof}[Proof of \Cref{lemma:stationary}]
{
Since $\tilde{\bV}$ must satisfy the first-order local optimality condition \eqref{eqn:firstorder}, that is, $\tilde{\bV}_i\bLambda_i = \sum_j \bS_{ij}\tilde{\bV}_j$, we can construct the block diagonal matrix $\bS^{(2)}$ by letting $\bS^{(2)}_{ii} = \tilde{\bV}_i\bLambda_i\tilde{\bV}_i^T = \sum_j \bS_{ij}\tilde{\bV}_j\tilde{\bV}_i^T$. Then it follows that
\[
    \Proj_{\tilde{\bV}_i}\bS^{(2)}_{ii}\Proj_{\tilde{\bV}_i}
    = \tilde{\bV}_i\bLambda_i\tilde{\bV}_i^T = \bS^{(2)}_{ii}
    .
\]
Furthermore, 
\[
    [\bS^{(2)}\tilde{\bV}]_i = \bS^{(2)}_{ii}\tilde{\bV}_i
    = \tilde{\bV}_i\bLambda_i
    = \sum_j \bS_{ij}\tilde{\bV}_j 
    = [\bS\tilde{\bV}]_i
    .
\]
Thus $\bS^{(2)}\tilde{\bV} = \bS\tilde{\bV}$ and for $\bS^{(1)} = \bS - \bS^{(2)}$, we see 
$\bS^{(1)}\tilde{\bV} = 0$ and $\tilde{\bV}^T\bS^{(1)}=0$ (by symmetry). This implies
$\Proj_{\tilde{\bV}^{\perp}}\bS^{(1)}\Proj_{\tilde{\bV}^{\perp}} = \bS^{(1)}$.
Hence condition \eqref{eq:KKT} is satisfied.
%
}
\end{proof}


\textbf{Step 2: Construction and verification of a certificate.}
We construct the certificates $\bT^{(1)}$ and $\bT^{(2)}$ based on $\bS^{(1)}$ and $\bS^{(2)}$ as follows:
\begin{equation}
\bT^{(1)}_{ij} =
\begin{cases}
\bS^{(1)}_{ij}, & i\neq j, \\
\bS^{(1)}_{ii}-c\Proj_{\tilde{\bV}_i^\perp}, & i=j
\end{cases}
,\qquad
\bT^{(2)}_{ij} = 
\begin{cases}
\bS^{(2)}_{ij}, & i\neq j, \\ \bS^{(2)}_{ii}-c\Proj_{\tilde{\bV}_i}, & i = j.
\end{cases}\label{eq:construction}
\end{equation}
It remains to verify that the certificate satisfies the assumptions in \Cref{lemma:equavalency}.\\ 
\textbf{Step 2a: proof of \eqref{eq:certificate1}}
From the properties of $\bS^{(1)}$ and $\bS^{(2)}$ from step 1, it is clear that  $\bS=\bT^{(1)}+\bT^{(2)}+c\bI$, $\Proj_{\tilde{\bV}^\perp}\bT^{(1)}\Proj_{\tilde{\bV}^\perp}=\bT^{(1)}$, and $\bT^{(2)}_{ii}=\Proj_{\tilde{\bV}_i}\bT^{(2)}_{ii} \Proj_{\tilde{\bV}_i}$. \\
\textbf{Step 2b: prove that $\{P_{\tilde{\bV}_i}^T\bT^{(2)}_{ii} P_{\tilde{\bV}_i}\}_{i=1}^m$ are  positive definite. }
Applying for all $1\leq i\leq m,$\begin{align}\label{eq:noisy2}
\|\tilde{\bV}_i^T[\bS^{(2)}]_{ii}\tilde{\bV}_i-m
\bI\|
&
{\leq}
\textstyle
\Big\|\sum_{j=1}^m\bW_{ij}\tilde{\bV}_i\Big\|+m\|\tilde{\bV}_i^T\bTheta_i-\bI\|+\|{\bTheta}^T\tilde{\bV}-m\bI\|
\end{align}
(which will be proved in step 3), 
{Weyl's inequality for perturbation of eigenvalues,} and noting that $\|\tilde{\bV}_i^T\bTheta_i-\bI\|\leq \|\tilde{\bV}_i-\bTheta_i\|$,
we see 
$P_{{\tilde{\bV}_i}}^T\bT^{(2)}_{ii}P_{{\tilde{\bV}_i}}$ is positive definite for all $1\leq i\leq m$ if
\begin{equation}\label{eq:noisy2_1}
m
>c+\max_{1\leq i\leq m}\Big\|{\textstyle\sum_{j=1}^m\bW_{ij}\tilde{\bV}_i}\Big\|+m\max_{1\leq i\leq m}\|\tilde{\bV}_i-\bTheta_i\|+\|{\bTheta}^T\tilde{\bV}-m\bI\|,
\end{equation}
which follows from \eqref{eq:noisy2_0} with $c=m/2$.\\
\textbf{Step 2c: prove that $-P_{\tilde{\bV}^\perp}^T\bT^{(1)}P_{\tilde{\bV}^\perp}$ is positive definite. }
Let $\Sp(\bX)$ be th column space of the matrix $\bX$, and  define the subspaces
$L_1=\Sp(\bTheta)$, $L_2=\{\bx\in\reals^D: \bx_i\in\Sp(\bTheta_i)\}$, and 
 $L_3=L_2^\perp=\{\bx\in\reals^D: \bx_i\in\Sp(\bTheta_i^\perp)\}$, and let $\bS^{(1)*}=-m\Proj_{L_2\cap L_1^\perp}$ and ${\bT^{(1)*}}=\bS^{(1)*}-c\Proj_{L_3}=-m\Proj_{L_2\cap L_1^\perp}-c\Proj_{L_3}$. More specifically, we have
 \begin{equation}\label{eq:tildeTstar10}
[\bS^{(1)*}]_{ij}=\bTheta_i\bTheta_j^T\,\,\text{for $i\neq j$}, [\bS^{(1)*}]_{ii}=-(m-1){\bTheta}_i{\bTheta}_i^T
\end{equation} 
and
 $\bT^{(1)*}$ as follows:
$
\bT^{(1)*}_{ij}=\bS^{(1)*}_{ij},\,\,\bT^{(1)*}_{ii}=\bS^{(1)*}_{ii}-c
{\Proj_{{\bTheta}_i^\perp}}
.
$ 

Considering that $\dim(L_2\cap L_1^\perp)=\dim(L_2)-\dim(L_1)=rm-r$ and $\dim(L_3)=D-\dim(L_2)=D-rm$, we have 
$\lambda_{r+1}(\bT^{(1)*})=-c.$
Applying Weyl's inequality 
and noting $\|\bTheta_i\bTheta_i^T-\tilde{\bV}_i\tilde{\bV}_i^T\|=\|\bTheta_i(\bTheta_i-\tilde{\bV}_i)^T+(\tilde{\bV}_i-\bTheta_i)\tilde{\bV}_i^T\|\leq 2\|\bTheta_i-\tilde{\bV}_i\|$, we have
\begin{align*}
|\lambda_{r+1}(\bT^{(1)*})  -\lambda_{r+1}(\bT^{(1)})|
& \leq 
\|\bT^{(1)*}-\bT^{(1)}\|
\\
&\leq 
\textstyle
\|{\bS}^{(1)*}-{\bS}^{(1)}\|+c\max_{1\leq i\leq m}
\|{\Proj_{{\bTheta}_i^\perp}}-\Proj_{\tilde{\bV}_i^\perp}\|
\\
&= 
\textstyle
\|{\bS}^{(1)*}-{\bS}^{(1)}\|+c\max_{1\leq i\leq m}\|\bTheta_i\bTheta_i^T-\tilde{\bV}_i\tilde{\bV}_i^T\|
\\
&\leq 
\textstyle
\|{\bS}^{(1)*}-{\bS}^{(1)}\|+2c\max_{1\leq i\leq m}\|\bTheta_i-\tilde{\bV}_i\|.
\end{align*}
Combining it with 
\begin{equation}\label{eq:noisy1}
\begin{split}
\|\bS^{(1)*}-\bS^{(1)}\|
&\leq 
m\max_{1\leq i\leq m}\|\tilde{\bV}_i-\bTheta_i\|
+\max_{1\leq i\leq m}\|{\textstyle \sum_{j=1}^m\bW_{ij}\tilde{\bV}_j}\|
\\
&\qquad +\|\bTheta^T\tilde{\bV}-m\bI\|
+ \|\bW\|
\end{split}
\end{equation}
(which will be proved in step 3) and 
\begin{equation}\label{eq:noisy2_2}c>(m+2c)\max_{1\leq i\leq m}\|\tilde{\bV}_i-\bTheta_i\|+\max_{1\leq i\leq m}\|\sum_{j=1}^m\bW_{ij}\tilde{\bV}_j\|+\|\bTheta^T\tilde{\bV}-m\bI\|
{+\|\bW\|}
\end{equation}
(which follows from \eqref{eq:noisy2_0} with $c=m/2$), $\lambda_{r+1}(\bT^{(1)})$ is negative, which means that $\bT^{(1)}$ has at least $D-r$ negative eigenvalues. By definition, $\bT^{(1)}$ has $r$ zero eigenvalues with eigenvectors spanning the column space of $\tilde{\bV}$, so $P_{\tilde{\bV}^\perp}^T\bT^{(1)}P_{\tilde{\bV}^\perp}$ is negative definite. 


\textbf{Step 3: proof of auxiliary inequalities \eqref{eq:noisy2} and \eqref{eq:noisy1}}\\
\noindent \textbf{Step 3a: proof of  \eqref{eq:noisy2}.} 
%
Combining 
equation \eqref{eq:tildeT2}
with 
\begin{equation}\label{eq:j_exclude_i}
\textstyle
\sum_{j=1}^{m}\bTheta_j^T\tilde{\bV}_j
= 
\bTheta^T\tilde{\bV}
,
\end{equation}
we see
\begin{align*}
\Big\|\tilde{\bV}_i^T & [\bS^{(2)}]_{ii}\tilde{\bV}_i-{m}\bI \Big\|
=
\textstyle
\Big\|\tilde{\bV}_i^T\Big({\sum_{j=1}^m}\bS_{ij}\tilde{\bV}_j\Big)-{m}\bI\Big\|
\\
& \leq
\textstyle
\Big\|\sum_{j=1}^m\bW_{ij}\tilde{\bV}_j\Big\|
+
\left\|\tilde{\bV}_i^T\bTheta_i \Big({\sum_{j=1}^m}{\bTheta}_j^T\tilde{\bV}_j\Big)-{m}\bI\right\|
\\
& \leq
\textstyle
\Big\|\sum_{j=1}^m\bW_{ij}\tilde{\bV}_i\Big\|
+
\Big\|\tilde{\bV}_i^T\bTheta_i-\bI\Big\|
\Big\|{\sum_{j=1}^m}{\bTheta}_j^T\tilde{\bV}_j\Big\|+\left\|{\sum_{j=1}^m}{\bTheta}_j^T\tilde{\bV}_j-{m}\bI\right\|
\\
& \leq
\textstyle
\Big\|\sum_{j=1}^m\bW_{ij}\tilde{\bV}_i\Big\|
+
m \Big\|\tilde{\bV}_i^T\bTheta_i-\bI \Big\|
+
\Big\|{\bTheta}^T\tilde{\bV}-m\bI\Big\|
,
\end{align*}
{where $\bS_{ij}=\bW_{ij}+\bTheta_i\bTheta_j^T$ when $i\neq j$ is used for the first inequality.}

\noindent\textbf{Step 3b: proof of \eqref{eq:noisy1}.}
Applying \eqref{eq:tildeT1}, \eqref{eq:tildeT2}, and \eqref{eq:tildeTstar10}, we have that for both MAXBET and MAXDIFF models,
\begin{equation}\label{eq:tildeTstar2}
[{\bS}^{(1)}-{\bS}^{(1)*}]_{ij} =
\begin{cases}
\bS_{ij}-\bTheta_i\bTheta_j^T=\bW_{ij}, & i\neq j, \\ 
{\bW_{ii}}
-({\sum_{j=1}^{m}}\bS_{ij}\tilde{\bV}_j)\tilde{\bV}_i^T+{m}\bTheta_i\bTheta_i^T,
& i = j.
\end{cases}
\end{equation}
As a result, 
\begin{align}\label{eq:pertubation_1}
\|\bS^{(1)}-\bS^{(1)*}\|\leq \|\bW\|
+
\max_{1\leq i\leq m}
\left\|\textstyle
\left({\sum_{j=1}^{m}}\bS_{ij}\tilde{\bV}_j\right)\tilde{\bV}_i^T-{m}\bTheta_i\bTheta_i^T
\right\|.
\end{align}

{Using \eqref{eq:j_exclude_i},}
we have 
\begin{equation}\label{eq:j_exclude_i2}
\begin{split}
& 
\textstyle
\|({\sum_{j=1}^{m}} \bTheta_i\bTheta_j^T\tilde{\bV}_j)\tilde{\bV}_i^T-{m}\bTheta_i\bTheta_i^T\|
= 
\|({\sum_{j=1}^{m}}\bTheta_j^T\tilde{\bV}_j)\tilde{\bV}_i^T-{m}\bTheta_i^T\|
\\
& \leq 
\textstyle
\|({\sum_{j=1}^{m}}\bTheta_j^T\tilde{\bV}_j)-{m}\bI\|+{m}\|\tilde{\bV}_i-\bTheta_i\|
{=} 
\|\bTheta^T\tilde{\bV}-m\bI\|+m\|\tilde{\bV}_i-\bTheta_i\|.
\end{split}
\end{equation}
Applying \eqref{eq:pertubation_1}, \eqref{eq:j_exclude_i2}, and $\bS_{ij}=\bW_{ij}+\bTheta_i\bTheta_j^T$ when $i\neq j$, \eqref{eq:noisy1} is proved.
\end{proof}

\begin{proof}[Proof of \Cref{lemma:pertubation}]
First, we remark that the choice of $\tilde{\bV}\in\reals^{D\times r}$ is only unique up to an $r\times r$ orthogonal matrix. That is, for any orthogonal matrix $\bO\in\reals^{r\times r}$, $\tilde{\bV}\bO$ is also a potential choice. In this proof, we choose $\tilde{\bV}$ such that $\bTheta^T\tilde{\bV}\in \reals^{r\times r}$ is a symmetric, positive semidefinite matrix, and as a result, $\tr(\bTheta^T\tilde{\bV})=\|\bTheta^T\tilde{\bV}\|_*$.

Then we have that
\begin{equation}\label{eq:pertubation1}
\begin{split}
\|\tilde{\bV}-\bTheta\|_F^2
&=
\textstyle
\sum_{i=1}^m\|\tilde{\bV}_i-\bTheta_i\|_F^2
=
\sum_{i=1}^m\|\tilde{\bV}_i\|_F^2+\|\bTheta_i\|_F^2-2\tr(\tilde{\bV}_i\bTheta_i^T)
\\
&=
\sum_{i=1}^m\|\tilde{\bV}_i\|_F^2+\|\bTheta_i\|_F^2-2\tr(\bTheta_i^T\tilde{\bV}_i)
=
2rm-2\tr(\sum_{i=1}^m\bTheta_i^T\tilde{\bV}_i)
\\
&=
2rm-2\tr(\bTheta^T\tilde{\bV})
=
2rm-2\|\bTheta^T\tilde{\bV}\|_*,
\end{split}
\end{equation}
where $\|\cdot\|_*$ represents the nuclear norm that is the summation of all singular values (and since $\bV^T\tilde{\bV}$ is positive semidefinite, it is also the summation of its eigenvalues).

Using the definition in \eqref{eq:tracesum_problem1}, we have
\begin{equation}\label{eq:pertubation2}
\tr(\tilde{\bV}^T\bS\tilde{\bV})\geq 
\tr({\bTheta}^T\bS{\bTheta}).
\end{equation}
Applying $\bS=\bTheta\bTheta^T+\bW$, \eqref{eq:pertubation2} implies
\begin{align*}
\tr(\tilde{\bV}^T\bW\tilde{\bV})+\|\tilde{\bV}^T\bTheta\|_F^2
&=
\tr(\tilde{\bV}^T\bW\tilde{\bV})+\tr(\tilde{\bV}^T\bTheta\bTheta^T\tilde{\bV})
\\
&\geq 
\tr(\bTheta^T\bW\bTheta)+\tr(\bTheta^T\bTheta\bTheta^T\bTheta)
=
\tr(\bTheta^T\bW\bTheta)+\|\bTheta^T\bTheta\|_F^2
\end{align*}
and
\begin{equation}\label{eq:pertubation4}
\tr(\tilde{\bV}^T\bW\tilde{\bV})-\tr({\bTheta}^T\bW{\bTheta})\geq \|\bTheta^T\bTheta\|_F^2-\|\tilde{\bV}^T\bTheta\|_F^2=rm^2-\|\tilde{\bV}^T\bTheta\|_F^2.
\end{equation}
Since $\|\bX\|_F^2=\sum_{i}\lambda_i(\bX)^2$, we have
\begin{equation}\label{eq:pertubation5}
\begin{split}
rm^2-\|\tilde{\bV}^T\bTheta\|_F^2
&=
\textstyle 
\sum_{i=1}^r(m^2-\lambda_i(\tilde{\bV}^T\bTheta)^2)
\\
&\geq
\textstyle
m \sum_{i=1}^r(m-\lambda_i(\tilde{\bV}^T\bTheta))
=
m(rm-\|\tilde{\bV}^T\bTheta\|_*).
\end{split}
\end{equation}
The combination \eqref{eq:pertubation4}, \eqref{eq:pertubation5}, $\|\tilde{\bV}\|_F=\|\bTheta\|_F=\sqrt{rm}$,  $\tr(\bA\bB)\leq \|\bA\|_F\|\bB\|_F$, and $\|\bA\bB\|_F\leq \|\bA\|\|\bB\|_F$ implies that
\begin{align*}
m(rm-\|\tilde{\bV}^T\bTheta\|_*)
&\leq 
\tr(\tilde{\bV}^T\bW\tilde{\bV})-\tr({\bTheta}^T\bW{\bTheta})
\\
&=
\tr((\tilde{\bV}-\bTheta)^T\bW\tilde{\bV})+\tr(\bTheta^T\bW(\tilde{\bV}-\bTheta))
\\
&\leq
\|\bW\|\|\tilde{\bV}-\bTheta\|_F\|\tilde{\bV}\|_F+\|\bW\|\|\tilde{\bV}-\bTheta\|_F\|{\bTheta}\|_F
\\
&=
2\|\bW\|\|\tilde{\bV}-\bTheta\|_F\sqrt{rm}.
\end{align*}
Combining it with \eqref{eq:pertubation1}, we have
\begin{align}\label{eq:pertubation6}
\textstyle
\frac{m}{2}\|\tilde{\bV}-\bTheta\|_F^2\leq 2\|\bW\|\|\tilde{\bV}-\bTheta\|_F\sqrt{rm},
\end{align}
which implies 
\begin{align}\label{eq:pertubation7}
\textstyle
\|\tilde{\bV}-\bTheta\|_F\leq 4\|\bW\|\sqrt{\frac{r}{m}},
\end{align}
proving the first inequality in \eqref{eq:lemma_pertubation1}. It implies that
\begin{align}\label{eq:pertubation7.5}
\|\tilde{\bV}^T\bTheta-m\bI\|_F
=
\|(\tilde{\bV}-\bTheta)^T\bTheta\|_F
\leq 
\|\tilde{\bV}-\bTheta\|_F\sqrt{m}
\leq 
4\|\bW\|\sqrt{r}.
\end{align}
Applying \eqref{eq:pertubation7}, the second inequality in \eqref{eq:lemma_pertubation1} is proved:
\begin{equation}\label{eq:pertubation13}
\begin{split}
\max_{1\leq i\leq m}\|[\bW\tilde{\bV}]_i\|_F
&\leq 
\max_{1\leq i\leq m}\|[\bW{\bTheta}]_i\|_F+\max_{1\leq i\leq m}\|[\bW(\tilde{\bV}-\bTheta)]_i\|_F
\\
&\leq 
\max_{1\leq i\leq m}\|[\bW{\bTheta}]_i\|_F + \|\bW\|\|\tilde{\bV}-\bTheta\|_F
\\
&\leq 
\max_{1\leq i\leq m}\|[\bW{\bTheta}]_i\|_F+4\|\bW\|^2
\textstyle
\sqrt{\frac{r}{m}}.
\end{split}
\end{equation}

Now let us consider $\bar{\bV}\in\reals^{D\times r}$ defined by $\bar{\bV}_i=\bTheta_i$ and $\bar{\bV}_j=\tilde{\bV}_j$ for all $1\leq j\leq m, j\neq i$. By definition we have
$
\tr(\tilde{\bV}^T\bS\tilde{\bV})\geq 
\tr(\bar{\bV}^T\bS\bar{\bV}),
$
and it is equivalent to
$
\tr((\tilde{\bV}-\bar{\bV})^T\bS\tilde{\bV})+
\tr( {\tilde{\bV}^T}\bS(\tilde{\bV}-\bar{\bV}))-\tr((\tilde{\bV}-\bar{\bV})^T\bS(\tilde{\bV}-\bar{\bV}))\geq 0.$
By the definition of $\bar{\bV}$, $\tilde{\bV}$, and $\bS$, we have
\begin{equation}\label{eq:pertubation_10}
\begin{split}
2\tr((\tilde{\bV}_i-\bTheta_i)^T\bTheta_i\bTheta^T\tilde{\bV})
& +
2\tr((\tilde{\bV}_i-\bTheta_i)^T[\bW\tilde{\bV}]_i)
\\
& -
\tr((\tilde{\bV}_i-\bTheta_i)^T\bS_{ii}(\tilde{\bV}_i-\bTheta_i))
\geq 
0.
\end{split}
\end{equation}

Recall that $\tilde{\bV}$ is chosen such that $\bTheta^T\tilde{\bV}$ is symmetric, positive semidefinite, and apply the fact that when $\bA$ is positive semidefinite, then $\tr(\bB\bA)=\tr(\bB^T\bA)$ and when both $\bA, \bB$ are p.s.d., $\tr(\bA\bB)\geq \tr(\bA\lambda_{\min}(\bB)\bI)\geq \lambda_{\min}(\bB)\tr(\bA)$ ($\lambda_{\min}$ represents the smallest eigenvalue), we have
\begin{equation}\label{eq:pertubation11}
\begin{split}
\tr & \big[(\bTheta_i  - \tilde{\bV}_i)^T  \bTheta_i\bTheta^T\tilde{\bV}\big]
=
\tr\big[(\bI-\tilde{\bV}_i^T\bTheta_i)(\bTheta^T\tilde{\bV})\big]
\\
&=
\textstyle
\frac{1}{2}\tr\big[(2\bI \!-\! \tilde{\bV}_i^T\bTheta_i \!-\! \bTheta_i^T\tilde{\bV}_i)(\bTheta^T\tilde{\bV})\big]
\!=\!
\frac{1}{2}\tr\big[(\tilde{\bV}_i \!-\! \bTheta_i)^T(\tilde{\bV}_i \!-\! \bTheta_i)(\bTheta^T\tilde{\bV})\big]
\\
\textstyle
&\geq 
\textstyle
\frac{1}{2}\tr\big[(\tilde{\bV}_i-\bTheta_i)^T(\tilde{\bV}_i-\bTheta_i)\big]\lambda_r(\bTheta^T\tilde{\bV})
=
\frac{1}{2}\|\tilde{\bV}_i-\bTheta_i\|_F^2\lambda_r(\bTheta^T\tilde{\bV}).
\end{split}
\end{equation}
In addition, we have
\begin{equation}\label{eq:pertubation_12}
\tr((\tilde{\bV}_i-\bTheta_i)^T\bS_{ii}(\tilde{\bV}_i-\bTheta_i))\geq -\|\bS_{ii}\|\|\tilde{\bV}_i-\bTheta_i\|_F^2\geq -(1+\|\bW_{ii}\|)\|\tilde{\bV}_i-\bTheta_i\|_F^2
\end{equation}
and $\tr(\bA\bB)\leq \|\bA\|_F\|\bB\|_F$ implies
\begin{equation}\label{eq:pertubation_14}
\tr((\tilde{\bV}_i-\bTheta_i)^T[\bW\tilde{\bV}]_i)\leq \|\tilde{\bV}_i-\bTheta_i\|_F\|\bW^T\tilde{\bV}\|_F. 
\end{equation}
Combining \eqref{eq:pertubation_10},  \eqref{eq:pertubation11}, \eqref{eq:pertubation_12}, and \eqref{eq:pertubation_14},
\[
\|\tilde{\bV}_i-\bTheta_i\|_F
 {\|[\bW\tilde{\bV}]_i\|_F}
\geq \|\tilde{\bV}_i-\bTheta_i\|_F^2
 {(\lambda_r(\bTheta^T\tilde{\bV})-1-\|\bW_{ii}\|)}
.
\]
Combining it with \eqref{eq:pertubation13} and \eqref{eq:pertubation7.5} which implies that $\lambda_r(\bTheta^T\tilde{\bV})\geq m-4\|\bW\|\sqrt{r}$,
 {and noting that $\|\bW\| \ge \|\bW_{ii}\|$,}
\eqref{eq:lemma_pertubation2} is proved.
\end{proof}
\subsubsection{\Cref{lemma:pertubation} under \eqref{eq:model2}}\label{sec:MAXBET}
\begin{proof}[Proof of \Cref{lemma:pertubation} under \eqref{eq:model2}]
Following the proof of \Cref{lemma:pertubation} under \eqref{eq:model}, we have
\begin{align*}
2\|\bW\|\|\tilde{\bV}-\bTheta\|_F\sqrt{rm}
&\geq 
\tr(\tilde{\bV}^T\bW\tilde{\bV})-\tr({\bTheta}^T\bW{\bTheta})
\\
&\geq 
\textstyle
(rm^2-\|\tilde{\bV}^T\bTheta\|_F^2)-(rm-\sum_{i=1}^m\|\tilde{\bV}_i^T\bTheta_i\|_F^2)
\\
&\geq
\textstyle
\frac{m}{2}\|\tilde{\bV}-\bTheta\|_F^2-\sum_{i=1}^m\|\tilde{\bV}_i-\bTheta_i\|_F^2
=
(\frac{m}{2}-1)\|\tilde{\bV}-\bTheta\|_F^2
,
\end{align*}
where the first inequality is \eqref{eq:pertubation5},
the second inequality is from the definition of $\bS$ under the MAXDIFF setting, and the third inequality is from
$
    r - \|\tilde{\bV}_i^T\bTheta_i\|_F^2
    \le \|\tilde{\bV}_i - \bTheta_i\|_F^2
    = 2r - \tr(\tilde{\bV}_i^T\bTheta_i)
    ,
$
since $\|\tilde{\bV}_i^T\bTheta\|_F^2 - 2\tr(\tilde{\bV}_i^T\bTheta) + r = \|\tilde{\bV}_i^T\bTheta_i - \bI_r\|_F^2$.

As a result,  {if $m > 2$},
\begin{align*}
&\|\tilde{\bV}-\bTheta\|_F
\leq 
4\|\bW\|
\textstyle\frac{\sqrt{rm}}{m-2},\,\,
\|\tilde{\bV}^T\bTheta- {m}\bI\|_F
\leq 
\textstyle4\|\bW\|\frac{m\sqrt{r}}{m-2}, \qquad
\\
&\max_{1\leq i\leq m}\|[\bW\tilde{\bV}]_i\|_F\leq
\max_{1\leq i\leq m}\|[\bW{\bTheta}]_i\|_F + 4\|\bW\|^2
\textstyle\frac{\sqrt{rm}}{m-2}.
\end{align*}
In addition, \eqref{eq:pertubation_10} is replaced with
$2\tr((\tilde{\bV}_i-\bTheta_i)^T\bTheta_i\bTheta^T\tilde{\bV})+2\tr((\tilde{\bV}_i-\bTheta_i)^T[\bW\tilde{\bV}]_i)
\allowbreak\geq\allowbreak 0.$
Then we have
$
\frac{1}{2}\lambda_r(\bTheta^T\tilde{\bV})\|\bTheta_i-\tilde{\bV}_i\|_F^2\leq \|\bTheta_i-\tilde{\bV}_i\|_F(\|[\bW\tilde{\bV}]_i\|_F)
$
and
\begin{equation*}
\max_{1\le i\le m}\|\bTheta_i-\tilde{\bV}_i\|_F\leq \frac{2\max_{1\leq i\leq m}\|[\bW{\bTheta}]_i\|_F + 8\|\bW\|^2\frac{\sqrt{rm}}{m-2}}{m - 4\|\bW\|\frac{m\sqrt{r}}{m-2}}
\end{equation*}
for $m > 4\|\bW\|\sqrt{r}+2$.
\end{proof}

\subsection{Proof of  lemmas for \Cref{thm:main2}}\label{sec:proofs}

\begin{proof}[Proof of \Cref{lem:weyllocal}]
    From $\bO_i\bLambda_i = \sum_{j=1}^m\bS_{ij}\bO_j$, we have $\bLambda_i = \sum_{j=1}^m\bO_i^T\bS_{ij}\bO_j$. 
    Hence, under \eqref{eq:model}, 
    \begin{align*}
        \|\bLambda_i - m\bI_r\|
        &= 
        \textstyle 
        \|\sum_{j=1}^m\bO_i^T\bS_{ij}\bO_j- m\bI_r\|
        \\
        &\leq
        \textstyle 
        \|\bO_i^T\sum_{j=1}^m\bW_{ij}\bO_j\|
        + \|\bO_i^T\bTheta_i\sum_{j=1}^m\bTheta_j^T\bO_j - m\bI_r\|
        \\
        &\leq
        \textstyle 
        \|\sum_{j=1}^m\bW_{ij}\bO_j\|
        + \|(\bO_i^T\bTheta_i - \bI_r)\sum_{j=1}^m\bTheta_j^T\bO_j + \sum_{j=1}^m\bTheta_j^T\bO_j - m\bI_r \|
        \\
        &\leq
        \|\sum_{j=1}^m\bW_{ij}\bO_j\|
        + \|(\bO_i^T\bTheta_i - \bI_r)\sum_{j=1}^m\bTheta_j^T\bO_j\| + \|\sum_{j=1}^m\bTheta_j^T\bO_j - m\bI_r \|
        \\        
        &\leq
        \textstyle 
        \|\sum_{j=1}^m\bW_{ij}\bO_j\|
        + \|\bO_i^T\bTheta_i - \bI_r\|\|\sum_{j=1}^m\bTheta_j^T\bO_j\| + \|\bTheta^T\bO - m\bI_r \|
        \\
        &\leq
        \textstyle 
        \|\sum_{j=1}^m\bW_{ij}\bO_j\|
        + m\|\bO_i^T\bTheta_i - \bI_r\| + \|\bTheta^T\bO - m\bI_r \|
    \end{align*}
    since $\|\bTheta_j^T\bO_j\| \le 1$.
    Under the MAXDIFF model,
    \begin{align*}
        \|\bLambda_i& - (m-1)\bI_r\|
        = 
        \textstyle
        \|\sum_{j\neq i}\bO_i^T\bS_{ij}\bO_j- (m-1)\bI_r\|
        \\
        &\leq
        \textstyle 
        \|\bO_i^T\sum_{j\neq i}\bW_{ij}\bO_j\|
        + \|\bO_i^T\bTheta_i\sum_{j \neq i}\bTheta_j^T\bO_j - (m-1)\bI_r\|
        \\
        &\leq
        \textstyle 
        \|\sum_{j \neq i}\bW_{ij}\bO_j\|
        + \|(\bO_i^T\bTheta_i - \bI_r)\sum_{j=1}^m\bTheta_j^T\bO_j + \sum_{j\neq i}\bTheta_j^T\bO_j - (m-1)\bI_r \|
        \\
        &\leq
        \textstyle 
        \|\sum_{j \neq i}\bW_{ij}\bO_j\|
        + \|(\bO_i^T\bTheta_i - \bI_r)\sum_{j \neq i}\bTheta_j^T\bO_j\| + \|\sum_{j \neq i}\bTheta_j^T\bO_j - (m-1)\bI_r \|
        \\        
        &\leq
        \|\sum_{j\neq i}\bW_{ij}\bO_j\|
        + \|\bO_i^T\bTheta_i - \bI_r\|\|\sum_{j \neq i}\bTheta_j^T\bO_j\| + \|\bTheta^T\bO - m\bI_r - \bTheta_i^T\bO_i + \bI_r \|
        \\
        &\leq
        \textstyle 
        \|\sum_{j \neq i}\bW_{ij}\bO_j\|
        + (m-1)\|\bO_i^T\bTheta_i - \bI_r\| + \|\bTheta^T\bO - m\bI_r \|
        + \|\bO_i^T\bTheta_i - \bI_r\|
        \\
        &=
        \textstyle 
        \|\sum_{j\neq i}\bW_{ij}\bO_j\|
        + m\|\bO_i^T\bTheta_i - \bI_r\| + \|\bTheta^T\bO - m\bI_r \|
        .
    \end{align*}    
\end{proof}

The following technical lemma is needed to prove \Cref{lem:upperboundlocal}.
\begin{lemma}\label{lem:normalcone}
	Suppose $\bX, \bY \in \mathcal{O}_{d,r}$ and $\bLambda \in \mathbb{R}^{d\times d}$ is symmetric and positive semidefinite.
	Then, there holds
	$\tr[ \bX\bLambda (\bY - \bX) ] \le 0$.
\end{lemma}
\begin{proof}
    Note
	\begin{align*}
		\tr[ \bX\bLambda (\bY - \bX) ]
		&\le
		\tr(\bLambda^T\bX^T(\bY-\bX))
	=
		\tr(\bLambda(\bX^T\bY - \bI_r))
		\\
		&= 
		\tr(\bLambda\bX^T\bY) - \tr(\bLambda)
		= 
		\tr(\bY^T\bX\bLambda) - \tr(\bLambda)
		\\
		&=
		\tr(\bLambda\bY^T\bX) - \tr(\bLambda)
	=
	    \textstyle
		\tr\left[\bLambda\left(\frac{1}{2}\bX^T\bY + \frac{1}{2}\bY^T\bX - \bI_r\right)\right]
        .
	\end{align*}
	Since $\bX\bX^T \psdle \bI_d$,
	$
		(\bX^T\bY)^T(\bX^T\bY) = \bY^T\bX\bX^T\bY
		\psdle \bY^T\bY = \bI_r
		.
	$
	Thus $\|\bX^T\bY\|_2 \le 1$.
	Likewise $\|\bY^T\bX\|_2 \le 1$. Then,
	because
	$\frac{1}{2}\bX^T\bY + \frac{1}{2}\bY^T\bX$
	is symmetric,
	$$
	\textstyle
	\lambda_{\max}\left(
		\frac{1}{2}\bX^T\bY + \frac{1}{2}\bY^T\bX
	\right)
	\le
	\left\| \frac{1}{2}\bX^T\bY + \frac{1}{2}\bY^T\bX \right\| 
	\le
	\frac{1}{2}\|\bX^T\bY\|_2 + \frac{1}{2}\|\bY^T\bX\| \le 1
	$$
	and
	$
		\frac{1}{2}\bX^T\bY + \frac{1}{2}\bY^T\bX - \bI_r \psdle \bzero
		.
	$
	Since $\bLambda \psdge \bzero$, it follows that
	$\tr[ \bX\bLambda (\bY - \bX) ] \le 0$.
\end{proof}

\begin{proof}[Proof of \Cref{lem:upperboundlocal}]
    Recall that $\bO$ is chosen without loss of generality such that $\bTheta^T\bO$ is symmetric and positive semidefinite. 
    
    Assumption \ref{assumption} plays a similar role to inequality \eqref{eq:pertubation2} in the proof of \Cref{lemma:pertubation}. It immediately follows that,
    under \eqref{eq:model},
    \[
    2\sqrt{mr}\|\bW\|\|\bO-\bTheta\|_F
    \ge
    \tr(\bO^T\bW\bO) - \tr(\bTheta^T\bW\bTheta)
    \ge
    \textstyle
    \frac{m}{2}\|\bO-\bTheta\|_F^2
    \]
    and under \eqref{eq:model2},
    \[
    2\sqrt{mr}\|\bW\|\|\bO-\bTheta\|_F
    \ge
    \tr(\bO^T\bW\bO) - \tr(\bTheta^T\bW\bTheta)
    \ge
    \textstyle
    \left(\frac{m}{2}-1\right)\|\bO-\bTheta\|_F^2
    ,
    \]
    from which inequality \eqref{eqn:localupperbound1} holds.
    Inequality \eqref{eqn:localupperbound2} follows from 
    \begin{align*}
    \|[\bW\bO]_i\|_F &\le
    \|[\bW(\bO - \bTheta)]_i\|_F + \|[\bW\bTheta]_i\|_F
    =
    \|\bW_{i\cdot}(\bO-\bTheta)\|_F + \|[\bW\bTheta]_i\|_F
    \\
    &\le
    \|\bW_{i\cdot}\|\|\bO - \bTheta\|_F + \|[\bW\bTheta]_i\|_F
    \le 
    \|\bW\|\|\bO - \bTheta\|_F + \|[\bW\bTheta]_i\|_F
    ,
    \end{align*}
    and inequality \eqref{eqn:localupperbound1},
    where $\bW_{i\cdot}=[\bW_{i1}, \dotsc, \bW_{im}]$ is the $i$th row block of $\bW$.
    
    Inequality \eqref{eqn:localupperbound1} also implies
    \begin{equation}\label{eqn:weyl}
    \|\bTheta^T\bO - m\bI_r\| \le
    \begin{cases}
    4\|\bW\|\sqrt{r}, &\text{under }\eqref{eq:model}, 
    \\
    4\|\bW\|\frac{\sqrt{r}}{1-2/m}, &\text{under }\eqref{eq:model2}. 
    \end{cases}
    \end{equation}
    
    We first consider the MAXBET model.
    Since $\bO=[\bO_1^T,\dotsc,\bO_m^T]^T$ is a
    {candidate}
    critical point, the associated Lagrange multiplier $\bLambda_i$ of $\bO_i$ satisfies $\bO_i\bLambda_i = \sum_{j=1}^m\bS_{ij}\bO_j$
	(see equation \eqref{eqn:firstorder})
	and is symmetric, positive semidefinite. 
	Since $\bS_{ij} = \bTheta_i\bTheta_j^T + \bW_{ij}$, 
	\begin{equation}\label{eqn:lagrange2}
		\sum_{j=1}^m\bS_{ij}\bO_j 
		=
		\sum_{j=1}^m\bTheta_i\bTheta_j^T\bO_j + \sum_{j=1}^m\bW_{ij}\bO_j
		=
		\bTheta_i\bTheta^T\bO + [\bW\bO]_i
		.
	\end{equation}
	Thus from \Cref{lem:normalcone}, we have
	\begin{equation}\label{eqn:normalbound}
	\begin{split}
	0 &\ge 
	\textstyle
	\tr[(\bTheta_i-\bO_i)^T\sum_{j\neq i}\bS_{ij}\bO_j] 
	\\
	&
	= \tr[(\bTheta_i-\bO_i)^T\bTheta_i\bTheta^T\bO]
	+ \tr[(\bTheta_i-\bO_i)^T[\bW\bO]_i]
	.
	\end{split}
	\end{equation}
	From Assumption \ref{assumption} and the choice of $\bO$, we have, similar to inequality \eqref{eq:pertubation11},
	\begin{equation}\label{eqn:eigenbound}
		\tr[(\bTheta_i-\bO_i)^T\bTheta_i\bTheta^T\bO]
		\ge
		\textstyle
		\frac{1}{2}\lambda_r(\bTheta^T\bO)\|\bTheta_i-\bO_i\|_F^2
		.
	\end{equation}
	Then the Cauchy-Schwarz inequality and inequality \eqref{eqn:localupperbound2} entail
	\[
	\textstyle
	\frac{1}{2}
	\displaystyle
	\lambda_r(\bTheta^T\bO)\max_{1\le i\le m}\|\bTheta_i-\bO_i\|_F
	\le \max_{1\le i\le m}\|[\bW\bO]_i\|_F
	\le \max_{1\le i\le m}\|[\bW\bTheta]_i\|_F + 4\|\bW\|^2
	\textstyle
	\frac{\sqrt{r}}{\sqrt{m}}
	.
	\]
	Combining inequality \eqref{eqn:weyl} and Weyl's inequalty,
	$
	\lambda_r(\bTheta^T\bO) \ge m - 4\|\bW\|\sqrt{r}
	$, 
	and inequality \eqref{eqn:localupperbound3} is obtained.
	
	Under the MAXDIFF model, equation \eqref{eqn:lagrange2} becomes
$		\bO_i\bLambda_i = \sum_{j\neq i}\bS_{ij}\bO_j 
		=
		\bTheta_i\bTheta^T\bO - \bTheta_i\bTheta_i^T\bO_i + [\bW\bO]_i$
	and inequality \eqref{eqn:normalbound} is replaced by
	\begin{align*}
	0 \ge
	\tr[(\bTheta_i-\bO_i)^T\bTheta_i\bTheta^T\bO]
	- \tr[(\bTheta_i-\bO_i)^T\bTheta_i\bTheta_i^T\bO_i]
	+ \tr[(\bTheta_i-\bO_i)^T[\bW\bO]_i]
	.
	\end{align*}
	Inequality \eqref{eqn:eigenbound} remains intact, and
	\begin{align*}
	    - \tr[(\bTheta_i \!-\! & \bO_i )^T  \bTheta_i\bTheta_i^T\bO_i]
	    \!=\! 
	    \tr[(\bO_i \!-\! \bTheta_i)^T\bTheta_i\bTheta_i^T(\bO_i-\bTheta_i)] - \tr[(\bTheta_i - \bO_i)^T\bTheta_i\bTheta_i^T\bTheta_i]
	    \\
	    &\ge
	    -\|\bTheta_i\bTheta_i^T\|\|\bO_i - \bTheta_i\|_F^2
	    - \tr[(\bTheta_i - \bO_i)^T\bTheta_i]
	    \\
	    &\ge
	    - \|\bO_i - \bTheta_i\|_F^2
	    - \tr(\bI_r - \bO_i^T\bTheta_i)
	    \\
	    &\ge
	    \textstyle
	    - \|\bO_i - \bTheta_i\|_F^2
	    -\frac{1}{2}[\|\bO_i\|_F^2 + \|\bTheta_i\|_F^2 - 2\tr(\bO_i^T\bTheta_i)]
	   =
	    -\frac{3}{2}\|\bO_i - \bTheta_i\|_F^2
	    ,
	\end{align*}
	where the third line is due to $\|\bTheta_i\bTheta_i^T\|\le 1$.
	Hence the Cauchy-Schwarz inequality and inequality \eqref{eqn:localupperbound2} now give
	\begin{align*}
	\textstyle
	\frac{1}{2}(\lambda_r(\bTheta^T\bO)-3)
	\displaystyle
	\max_{1\le i\le m}\|\bTheta_i-\bO_i\|_F
	&\le 
	\max_{1\le i\le m}\|[\bW\bO]_i\|_F
	\\
	&\le 
	\max_{1\le i\le m}\|[\bW\bTheta]_i\|_F + 4\|\bW\|^2
	\textstyle
	\frac{\sqrt{r}}{\sqrt{m}-2/\sqrt{m}}
	.
	\end{align*}
	Inequality \eqref{eqn:weyl} and Weyl's inequality now result in
	$
	\lambda_r(\bTheta^T\bO) \ge m - 4\|\bW\|\frac{\sqrt{r}}{1-2/m}
	$
	and inequality \eqref{eqn:localupperbound3} is obtained.
	For a valid bound we need
	$
	m > 4\|\bW\|\frac{\sqrt{r}}{1-2/m} + 3
	.
	$
	Solving the involved quadratic inequality provides the desired lower bound for $m$.
\end{proof}

\section{Conclusion}
This paper studies the orthogonal trace-sum maximization \cite{won2018orthogonal}. It shows two results  when the noise is small: first, that while the problem is nonconvex, its solution can be achieved by solving its convex relaxation; second, 
{a critical point with the qualification of \Cref{def:qual} and \Cref{assumption} is also its global minimizer with high probability.}

A future direction is to improve the estimation on maximum noise that this method can handle. While this paper shows that the method succeeds when $\sigma=O(m^{1/4})$, we expect that it would also hold for noise as large as $\sigma=O(m^{1/2})$, which has been proved in \cite{ZhongBoumal18PhaseSynchNearOptimalBounds} for phase synchronization and in \cite{ling2021improved} for synchronization of rotations. {We suspect that the suboptimality of this result arises from the estimation of  $\max_{1\leq i\leq m} ||\sum_{j=1}^m \bW_{ij} \tilde{\bV}_{j}||$ in \eqref{eq:lemma_pertubation1}, where standard  tools from the theory of {measure concentration} can not be used as $\tilde{\bV}$ depends on $\bW$. To solve this problem, some decoupling techniques in probability theory might be needed to disentangle the dependence structure. }Another future direction is to use a more generic model than the additive Gaussian noise model, which would have a larger range of real-life applications.


\appendix
\section{Simulation study for \Cref{assumption}}

In order to see how often \Cref{assumption} is satisfied in working local optimization algorithms for \eqref{eqn:tracemax}, we conducted a simulation study. 
Under the data generation model \eqref{eq:model},
we fixed  $d=5$, $r=3$, and vary the number of groups $m\in \{2,5,10\}$  and the noise level  $\sigma \in \{0.01,0.1,1,10\}$. 
The semi-orthogonal matrices $\bTheta_1, \dotsc, \bTheta_m$ were generated by taking the QR decomposition of random $d \times r$ matrices with i.i.d. standard normal entries.
The upper triangular part including the diagonal of $\bW$ was generated from i.i.d. normal with mean zero and variance $\sigma^2$. For each combination of $m$ and $\sigma$, we generated $100$ replicates and report the number of replicates for which the proximal block ascent algorithm in \cite{won2018orthogonal} produces a qualified candidate critical point satisfying \Cref{assumption} using the Ten Berge initialization strategy (``tb" in \cite{won2018orthogonal}) in \Cref{tbl:sim}.
In addition, we also count the frequency of satisfying conditions \eqref{eqn:noiselevel} and \eqref{eq:cor:discordant2}
for \Cref{eq:cor:discordant2,cor:gaussian}, respectively,
and the certificate of global optimality of a critical point \eqref{eqn:certificate}.

\Cref{tbl:sim} shows that \Cref{assumption} is satisfied 100\% of times for all tried combinations. As expected, condition \eqref{eqn:noiselevel} is satisfied at small noise levels. Condition \eqref{eq:cor:discordant2}, which is fully determined by the combination of $m$ and $\sigma$, is less frequently satisfied than \eqref{eqn:noiselevel}.
In case either condition \eqref{eqn:noiselevel} or \eqref{eq:cor:discordant2} is satisfied, the certificate \eqref{eqn:certificate} is always satisfied because \Cref{assumption} holds all the time in this simulation. It is remarkable that certificate \eqref{eqn:certificate} is satisfied more frequently than condition \eqref{eqn:noiselevel} or \eqref{eq:cor:discordant2}.

\begin{table}
\begin{center}
\caption{Frequency of satisfaction of \Cref{assumption}, conditions \eqref{eqn:noiselevel}, \eqref{eq:cor:discordant2}, and certificate \eqref{eqn:certificate}}\label{tbl:sim}
\begin{tabular}{rrrrcr}
\toprule
$m$ & $\sigma$ & \Cref{assumption}\textsuperscript{\dag} &  \eqref{eqn:noiselevel}\textsuperscript{\dag} &  \eqref{eq:cor:discordant2} &  \eqref{eqn:certificate}\textsuperscript{\dag} \\ 
\midrule
\multirow{4}{*}{10} &
0.01 & 100 & 100 & TRUE & 100 \\ 
&
0.10 & 100 & 10 & FALSE & 100 \\ 
&
1.00 & 100 & 0 & FALSE & 0 \\ 
&
1.50 & 100 & 0 & FALSE & 0 \\ 
\midrule
\multirow{4}{*}{20} &
0.01 & 100 & 100 & TRUE & 100 \\ 
&
0.10 & 100 & 0 & FALSE & 100 \\ 
&
1.00 & 100 & 0 & FALSE & 21 \\ 
&
1.50 & 100 & 0 & FALSE & 0 \\ 
\midrule
\multirow{4}{*}{30} &
0.01 & 100 & 100 & TRUE & 100 \\ 
&
0.10 & 100 & 0 & FALSE & 100 \\ 
&
1.00 & 100 & 0 & FALSE & 99 \\ 
&
1.50 & 100 & 0 & FALSE & 0 \\ 
\bottomrule
\\
\end{tabular}
\begin{minipage}{\linewidth}
\centering
\dag Reported numbers are out of 100 replicates in each scenario.\\ 
\end{minipage}
\end{center}
\end{table}

\bibliographystyle{siamplain}
\bibliography{certificate}

\end{document}